\documentclass[12pt,leqno]{amsart}
\usepackage{enumerate}
\usepackage{amsrefs}
\usepackage{amsfonts, amsmath, amssymb, amscd, amsthm, bm, cancel}
\usepackage{url}
\usepackage{graphicx}
\usepackage[
linktocpage=true,colorlinks,citecolor=magenta,linkcolor=blue,urlcolor=magenta]{hyperref}
\usepackage{multicol}
\usepackage{comment}
\usepackage[margin=1in]{geometry}
\makeatletter
\@namedef{subjclassname@2020}{\textup{2020} Mathematics Subject Classification}
\makeatother
 \newtheorem{thm}{Theorem}[section]
 \newtheorem{Rem}{Remark}[section]
  \newtheorem{Cor}{Corollary}[section]

 \newtheorem{cor}[Cor]{Corollary}
  
 \newtheorem{lem}[thm]{Lemma}
 
 \newtheorem{prop}[thm]{Proposition}
 \theoremstyle{definition}
 \newtheorem{defn}[thm]{Definition}
  \newtheorem*{ack}{Acknowledgments}
 \theoremstyle{remark}
 \newtheorem{rem}[Rem]{Remark} 

 \numberwithin{equation}{section}
 \allowdisplaybreaks

\begin{document}
\title[Uniqueness of solutions to curvature problems]{Uniqueness of solutions to some classes of anisotropic and isotropic curvature problems}

\author{Haizhong Li}
\address{Department of Mathematical Sciences, Tsinghua University, Beijing 100084, P.R. China}
\email{\href{mailto:lihz@tsinghua.edu.cn}{lihz@tsinghua.edu.cn}}

\author{Yao Wan}
\address{Department of Mathematical Sciences, Tsinghua University, Beijing 100084, P.R. China}
\email{\href{mailto:y-wan19@mails.tsinghua.edu.cn}{y-wan19@mails.tsinghua.edu.cn}}



\begin{abstract}
In this paper, we apply various methods to establish the uniqueness of solutions to some classes of anisotropic and isotropic curvature problems. Firstly, by employing integral formulas derived by S. S. Chern \cite{Ch59}, we obtain the uniqueness of smooth admissible solutions to a class of Orlicz-(Christoffel)-Minkowski problems.
Secondly, inspired by Simon's uniqueness result \cite{Si67}, we then prove that the only smooth strictly convex solution to the following isotropic curvature problem
\begin{equation}\label{ab-1}
     \left(\frac{P_k(W)}{P_l(W)}\right)^{\frac{1}{k-l}}=\psi(u,r)\quad \text{on}\ \mathbb{S}^n
\end{equation}
must be an origin-centred sphere, where $W=(\nabla^2 u+u g_0)$, $\partial_1\psi\ge 0,\partial_2\psi\ge 0$ and at least one of these inequalities is strict. As an application, we establish the uniqueness of solutions to the isotropic Gaussian-Minkowski problem. Finally, we derive the uniqueness result for the following isotropic $L_p$ dual Minkowski problem
\begin{equation}\label{ab-2}
     u^{1-p} r^{q-n-1}\det(W)=1\quad \text{on}\ \mathbb{S}^n,
\end{equation}
where $-n-1<p\le -1$ and $n+1\le q\le n+\frac{1}{2}+\sqrt{\frac{1}{4}-\frac{(1+p)(n+1+p)}{n(n+2)}}$. This result utilizes the method developed by Ivaki and Milman \cite{IM23} and generalizes a result due to Brendle, Choi and Daskalopoulos \cite{BCD17}.
\end{abstract}

\maketitle

\section{Introduction}\label{sec:1}

Let $M_{0}$ be a smooth closed strictly convex hypersurface in Euclidean space $\mathbb{R}^{n+1}$ which encloses the origin, and let $X_{0}: \mathbb{S}^{n} \rightarrow \mathbb{R}^{n+1}$ be a smooth immersion with $X_{0}(\mathbb{S}^{n})=M_0$. Suppose $f$ is a positive continuous function defined on $\mathbb{S}^n$, $\varphi$ is a positive continuous function defined in $(0,\infty)$, and $k$ is an integer with $1\leq k \leq n$. We consider the anisotropic expanding curvature flow given by a family of smooth immersions $X:\mathbb{S}^{n}\times[0,T) \rightarrow \mathbb{R}^{n+1}$ solving the evolution equation
\begin{equation}\label{ev-eq}
\begin{split}
    \left\{\begin{array}{l}
    \frac{\partial X}{\partial t} (\cdot,t)= \frac{1}{f(\nu)}P_{k}(\cdot,t) \varphi(\langle X, \nu \rangle)\langle X, \nu \rangle \nu,\\
    X(\cdot, 0) = X_{0}(\cdot),
    \end{array}\right.
\end{split}
\end{equation}
where $\nu$ is the unit outer normal vector of $M_t=X(\mathbb{S}^n,t)$, and $P_k$ is the normalized $k$-th elementary symmetric function for principal curvature radii, i.e.
$$P_k(\cdot,t)=\frac{1}{\binom{n}{k}}\sum\limits_{i_1<\ldots<i_k}\lambda_{i_1}\cdots\lambda_{i_k},$$
$\lambda_i$ is the principal curvature radii of $M_{t}$. Let $(\mathbb{S}^n,g_0,\nabla)$ be the unit sphere equipped with its standard metric and Levi-Civita connection. For a smooth closed strictly convex hypersurface $M$ with the support function $u=\langle X,\nu\rangle$, we note that principal curvature radii $(\lambda_1,\ldots,\lambda_n)$ of $M$ are the eigenvalues of the positive-definite matrix $W=(\nabla^2 u+u g_0)$.
We call $X: \mathbb{S}^{n} \rightarrow \mathbb{R}^{n+1}$ a \textit{homothetic self-similar solution} to the flow (\ref{ev-eq}) if it satisfies
\begin{equation}\label{self-eq}
    c\varphi(u)P_k(W)=f\quad \text{on}\ \mathbb{S}^n,
\end{equation}
where $c$ is a positive constant and $u$ is the support function of $M=X(\mathbb{S}^n)$. In this paper, we are concerned with the uniqueness of solutions to (\ref{self-eq}).


When $k=n$, (\ref{self-eq}) can be written as a Monge-Amp\`ere type equation
\begin{equation}\label{OM-eq}
    c\varphi(u)\det(W)=f\quad \text{on}\ \mathbb{S}^n,
\end{equation}
which corresponds to the smooth case of the \textit{Orlicz-Minkowski problem}. The Orlicz-Minkowski problem is a fundamental problem in convex geometry within the framework of the Orlicz-Brunn-Minkowski theory. It serves as a generalization of the classical Minkowski problem, which investigates the necessary and sufficient conditions for a Borel measure on the unit sphere $\mathbb{S}^n$ to be a multiple of the Orlicz surface area measure of a convex body in $\mathbb{R}^{n+1}$. In \cite{HLYZ10}, Haberl, Lutwak, Yang and Zhang studied the even case of the Orlicz-Minkowski problem. Since then, this problem has attracted significant attention from numerous scholars, see e.g. \cites{BBC19,BIS21,GHWXY19,GHXY20,HH12,HLYZ16,JL19,Sa22,Sun18,SL15,XJL14}.
In particular, when $\varphi(s)=s^{1-p}$,  (\ref{OM-eq}) reduces to the \textit{$L_p$ Minkowski problem}, which has been extensively studied in the literature (see, e.g. \cites{And03,Bor22,BLYZ13, CW06,CY76,HLW16,HLX15,HLYZ05,JLW15,JLZ16,KM22,LW13,Lut93,Mil21,Zhu14} and Schneider's book \cite{Sch14} for relevant references). 
In the smooth case, the $L_p$ Minkowski problem is equivalent to solving the following PDE:
\begin{equation}\label{pM-eq}
    u^{1-p}\det(W)=f\quad \text{on}\ \mathbb{S}^n.
\end{equation}

When $1\leq k\le n-1$, (\ref{self-eq}) is known as the \textit{Orlicz-Christoffel-Minkowski problem}, see e.g. \cite{JLL21}. Moreover, when $1\le k\le n-1$ and $\varphi(s) = s^{1-p}$, (\ref{self-eq}) is referred to as the \textit{$L_{p}$ Christoffel-Minkowski problem}, which can be reduced to the following nonlinear PDE:
\begin{equation}\label{pCM-eq}
    u^{1-p} P_k(W)=f\quad \text{on}\ \mathbb{S}^n.
\end{equation}
We refer to \cites{Ch20,GLM18,Gu10,GX18,HMS04,Iv19,SY20} as examples for investigating the existence, uniqueness, and convexity of solutions to (\ref{pCM-eq}).

$\ $

On the other hand, let $\Tilde{M}_{0}$ be a smooth closed strictly convex hypersurface in $\mathbb{R}^{n+1}$ which encloses the origin, and let $\Tilde{X}_{0}: \mathbb{S}^{n} \rightarrow \mathbb{R}^{n+1}$ be a smooth immersion with $\Tilde{X}_{0}(\mathbb{S}^{n})=\Tilde{M}_0$. Suppose $\Psi$ is a positive continuous function defined in $(0,\infty)\times(0,\infty)$, and $l,k$ are two integer with $0\leq l< k \leq n$. We consider the isotropic expanding curvature flow given by a family of smooth immersions $\Tilde{X}:\mathbb{S}^{n}\times[0,T) \rightarrow \mathbb{R}^{n+1}$ solving the evolution equation
\begin{equation}\label{ev-eq-2}
\begin{split}
    \left\{\begin{array}{l}
    \frac{\partial \Tilde{X}}{\partial t} (\cdot,t)= \left(\frac{P_{k}}{P_{l}}\right)^{\frac{1}{k-l}}(\cdot,t) \Psi(\langle \Tilde{X}, \Tilde{\nu} \rangle,|\Tilde{X}|)\langle \Tilde{X}, \Tilde{\nu} \rangle \Tilde{\nu},\\
    \Tilde{X}(\cdot, 0) = \Tilde{X}_{0}(\cdot),
    \end{array}\right.
\end{split}
\end{equation}
where $\Tilde{\nu}$ is the unit outer normal vector of $\Tilde{M}_t=\Tilde{X}(\mathbb{S}^n,t)$. We call $\Tilde{X}: \mathbb{S}^{n} \rightarrow \mathbb{R}^{n+1}$ a \textit{homothetic self-similar solution} to the flow (\ref{ev-eq-2}) if it satisfies
\begin{equation}\label{self-eq-2}
    c\Psi(u,r)\left(\frac{P_k}{P_l}\right)^{\frac{1}{k-l}}(W)=1\quad \text{on}\ \mathbb{S}^n,
\end{equation}
where $c$ is a positive constant, $u$ is the support function and $r$ is the radial function of $\Tilde{M}=\Tilde{X}(\mathbb{S}^n)$. When $l=0$ and $\partial_2\psi=0$, then (\ref{self-eq-2}) reduces to the isotropic case of (\ref{self-eq}). Moreover, there are two important cases for (\ref{self-eq-2}): First, when $l=0,\ k=n$ and $\Psi(s,t)=s^{1-p}t^{q-n-1}$, (\ref{self-eq-2}) becomes the isotropic $L_p$ dual Minkowski problem, corresponding to the following equation
\begin{equation}\label{p,q-M-eq}
    u^{1-p} r^{q-n-1}\det(W)=1\quad \text{on}\ \mathbb{S}^n.
\end{equation}
Second, when $l=0,\ k=n$ and $\Psi(s,t)=s^{1-p}e^{-\frac{t^2}{2}}$, (\ref{self-eq-2}) becomes the isotropic $L_p$ Gaussian-Minkowski problem, corresponding to the following equation
\begin{equation}\label{p-GM-eq}
    u^{1-p} e^{-\frac{r^2}{2}}\det(W)=1\quad \text{on}\ \mathbb{S}^n.
\end{equation}

We are also concerned with the uniqueness of solutions to isotropic curvature problems. When $\Psi(u,r)=u^{\alpha}$, we refer to \cites{CG21,GLW22,GLM18} regarding the uniqueness of solutions to (\ref{self-eq-2}).


\subsection{Uniqueness of solutions to Orlicz-(Christoffel)-Minkowski problems}$\ $

In this paper, we always assume that all hypersurfaces $M^n$ are smooth closed hypersurfaces in $\mathbb{R}^{n+1}$ that contain the origin in its interior and bound a convex body $K$ with strictly positive Gauss curvature. Denote by $u$ the support function of $K$, and by $r$ the radial function of $K$. We write the Hessian matrix of $u$ with respect to $(\mathbb{S}^n,g_0,\nabla)$ as $W=(\nabla^2 u+u g_0)$.

Let $\varphi:(0,\infty)\to (0,\infty)$ be a continuous function. We consider the following Orlicz-Christoffel-Minkowski problem:
\begin{equation}\label{OCM-eq}
 \varphi(u)P_k(W)=f\quad \text{on}\ \mathbb{S}^n.
\end{equation}
A solution $u$ to (\ref{OCM-eq}) is called \textit{admissible} if $W=(\nabla^2 u+u g_0)\in\Gamma_k$, where $\text{Sym}(n)$ is the space of $n\times n$ symmetric matrices and $\Gamma_k=\{A\in\text{Sym}(n):\ P_i(A)>0,\ i=1,\ldots,k\}$.

Using integral formulas, S. S. Chern \cite{Ch59} proved the classical Alexandrov-Fenchel-Jessen theorem, which establishes the uniqueness of solutions to (\ref{OCM-eq}) with $\varphi\equiv 1$. Furthermore, we obtain the following uniqueness result by utilizing integral formulas.

\begin{thm}\label{main-thm-1}
Let $2\le k\le n$ be an integer. Let $\varphi:(0,\infty)\to(0,\infty)$ be a continuous function that satisfies
\begin{equation}\label{1.11}
\begin{split}
    & (s\varphi^{\frac{1}{k}}(s)-t\varphi^{\frac{1}{k}}(t))(\varphi^{\frac{k-1}{k}}(s)-\varphi^{\frac{k-1}{k}}(t))<0,
\end{split}
\end{equation}
for any $t>s>0$. For any positive function $f\in C(\mathbb{S}^n)$, then the smooth admissible solution to (\ref{OCM-eq}) is unique.
\end{thm}

In particular, when considering the $L_p$ Christoffel-Minkowski problem (\ref{pCM-eq}), the uniqueness of admissible solutions was established by Guan-Ma-Trudinger-Zhu \cite{Gu10} for $p=1$, by Hu-Ma-Shen \cite{HMS04} for $p\ge k+1$, and Guan-Xia \cite{GX18} for $1<p<k+1$. In the isotropic case, Chen \cite{Ch20} proved the uniqueness of strictly convex solutions to (\ref{pCM-eq}) with $f\equiv 1$ for $1>p>1-k$, while McCoy \cite{Mc11} established the uniqueness for $p=1-k$. As an application of Theorem \ref{main-thm-1}, by choosing $\varphi(s)=s^{1-p}$, we provide a new proof for the following uniqueness result. 

\begin{cor}[\cite{GX18}]\label{cor-2}
Let $2\le k\le n$ be an integer. Suppose $1<p<k+1$ is a real number. For any positive function $f\in C(\mathbb{S}^n)$, then the smooth admissible solution to (\ref{pCM-eq}) is unique. 
\end{cor}

Here are some more examples of functions that satisfy the condition (\ref{1.11}):
\begin{enumerate}[(1)]
    \item $\varphi(s)=\mu s^{\alpha}+s^{\beta}$, where $0\ge \alpha>\beta> -k$ and $\mu\ge0$.
    
    \item $\varphi(s)=(\mu s^{\alpha}+s^{\beta})^{-1}$, where $0\le \alpha<\beta< k$ and $\mu\ge0$.
    
    \item $\varphi(s)=e^{-s}$ with the assumption $0<u<k$.

    \item $\varphi(s)=s^{1-p_1}$, if $0<s\le 1$; $\varphi(s)=s^{1-p_2}$, if $s> 1$, where $1<p_1,p_2<k+1$.
\end{enumerate}


\subsection{Uniqueness of solutions to a class of isotropic curvature problems}$\ $

Let $\psi:(0,\infty)\times(0,\infty)\to (0,\infty)$ be a continuous function, and $l,k$ be two integers with $0\leq l< k \leq n$. We consider the following isotropic curvature problem:
\begin{equation}\label{icp-eq}
    \left(\frac{P_k(W)}{P_l(W)}\right)^{\frac{1}{k-l}}=\psi(u,r)\quad \text{on}\ \mathbb{S}^n.
\end{equation}
After giving new integral formulas involving the radial function, we obtain the following theorem:
\begin{thm}\label{mainthm-2}
Let $n\ge 2$ and $0\le l< k\le n$. Suppose $\psi:(0,\infty)\times(0,\infty)\to (0,\infty)$ is a continuous function, and there exists a $C^1$ decreasing function $\eta:(0,\infty)\to(0,\infty)$ satisfying 
\begin{align}\label{cond-3}
    \left(s\psi(1,1)-\psi(s,t)\right)\left(\eta(1)\psi^{k-l-1}(1,1)-\eta(t)\psi^{k-l-1}(s,t)\right)\le 0,
\end{align}
for any $t> s>0$. Then the smooth strictly convex solution $M$ to (\ref{icp-eq}) must be a sphere. Moreover, if in addition $\eta'(t)<0$ for any $t>0$, then $M$ is an origin-centred sphere.
\end{thm}

Let $\psi(u,r)=u^{\alpha}\phi(r)$, we consider the following equation
\begin{equation}\label{icp-eq-2}
    \left(\frac{P_{k}(W)}{P_{l}(W)}\right)^{\frac{1}{k-l}}=u^{\alpha}\phi(r) \quad \text{on}\ \mathbb{S}^n.
\end{equation}
Note that if $0\le \alpha<1$ and $\phi$ is increasing, then $\psi(u,r)$ satisfies the condition (\ref{cond-3}) with $\eta(t)=\phi^{-\frac{k-l-1}{1-\alpha}}(t)$. Therefore, we obtain

\begin{cor}\label{cor-3}
Let $n\ge 2$ and $0\le l< k\le n$. Suppose $0\le \alpha <1$ is a real number, and $\phi:(0,\infty)\to (0,\infty)$ is a $C^1$ increasing function. Then the smooth strictly convex solution $M$ to (\ref{icp-eq-2}) must be a sphere. Moreover, if in addition $\phi'(t)>0$ for any $t>0$, then $M$ is an origin-centred sphere.
\end{cor}


$\ $

On the other hand, Simon \cite{Si67} studied (\ref{icp-eq}) when $\partial_2\psi=0$ and proved some uniqueness theorems, including the classical Alexandrov-Fenchel-Jessen theorem. Inspired by Simon's result, we obtain the following uniqueness theorem for (\ref{icp-eq}):

\begin{thm}\label{mainthm-3}
Let $n\ge 2$ and $0\le l< k\le n$. Suppose $\psi:(0,\infty)\times(0,\infty)\to (0,\infty)$ is a $C^1$ function with $\partial_1\psi\ge 0,\ \partial_2\psi\ge 0$ and at least one of these inequalities is strict. Then the smooth strictly convex solution $M$ to (\ref{icp-eq}) must be an origin-centred sphere.
\end{thm}

When $l=0$, we consider the following isotropic curvature problem:
\begin{equation}\label{icp-eq-3}
    P_k(W)=\psi(u,r) \quad \text{on}\ \mathbb{S}^n.
\end{equation}
Then the following corollary is a direct consequence of Theorem \ref{mainthm-3}.
\begin{cor}\label{cor-4}
Let $n\ge 2$ and $1\le k\le n$. Suppose $\psi:(0,\infty)\times(0,\infty)\to (0,\infty)$ is a $C^1$ function with $\partial_1\psi\ge 0,\ \partial_2\psi\ge 0$ and at least one of these inequalities is strict. Then the smooth strictly convex solution $M$ to (\ref{icp-eq-3}) must be an origin-centred sphere.
\end{cor}

In particular, when $l=0$ and $k=n$, we provide a new proof of the uniqueness result obtained by Ivaki \cite{Iv23} for a large class of Minkowski type problems in the case $n\ge 2$.
\begin{cor}[\cite{Iv23}]\label{cor-5}
Let $n\ge 2$. Suppose $\psi:(0,\infty)\times(0,\infty)\to (0,\infty)$ is a $C^1$ function with $\partial_1\psi\ge 0,\ \partial_2\psi\ge 0$ and at least one of these inequalities is strict. Then the smooth strictly convex solution $M$ to the following isotropic Minkowski type problem
\begin{equation}\label{icp-eq-4}
   \det(W)=\psi(u,r)\quad \text{on}\ \mathbb{S}^n
\end{equation}
must be an origin-centred sphere.
\end{cor}

As applications, we can immediately deduce the following corollaries:
\begin{cor}\label{cor-6}
Let $n\ge 2$ and $1\le k\le n$. Suppose $p\ge 1,\ q\le k+1$ and at least one of these inequalities is strict. Then the smooth strictly convex solution $M$ to the following isotropic $L_p$ dual Christoffel-Minkowski problem 
\begin{equation}\label{i-pq-CM}
    u^{1-p} r^{q-k-1}P_k(W)=1\quad \text{on}\ \mathbb{S}^n
\end{equation}
must be an origin-centred sphere.
\end{cor}

\begin{rem}
When $k=n$, Corollary \ref{cor-6} reduces to the uniqueness theorem of the isotropic $L_p$ dual Minkowski problem (\ref{p,q-M-eq}) for $p>1$ and $q\le n+1$, as shown by Chen-Li \cite{CL21}.
\end{rem}

\begin{cor}\label{cor-7}
Let $n\ge 2$ and $1\le k\le n$. Suppose $p\ge 1$ is a real number. Then the smooth strictly convex solution $M$ to the following isotropic $L_p$ Gaussian Christoffel-Minkowski problem
\begin{equation}\label{i-p-GCM}
    u^{1-p} e^{-\frac{r^2}{2}} P_k(W)=1\quad \text{on}\ \mathbb{S}^n
\end{equation}
must be an origin-centred sphere.
\end{cor}

\begin{rem}
Corollary \ref{cor-7} provides an affirmative answer to the conjecture proposed by Chen-Hu-Liu-Zhao \cite{CHLZ23}, and another proof can be found in \cite{Iv23}.
\end{rem}


\subsection{Uniqueness of solutions to $L_p$ dual Minkowski problems}$\ $

The $L_p$ dual Minkowski problem was posed by Lutwak, Yang and Zhang in \cite{LYZ18}, and unifies the $L_p$ Minkowski problem and the dual Minkowski problem. The $L_p$ dual Minkowski problem has been intensively studied by many authors, see e.g. \cites{BF19, CCL21, CHZ19, CL21, HZ18, LW22, LLL22, SX21}. In the smooth case, given a positive continuous function $f$ on $\mathbb{S}^{n}$, the $L_p$ dual Minkowski problem becomes the following Monge-Amp\`ere type equation:
\begin{equation}\label{an-p,q-M-eq}
    u^{1-p}r^{q-n-1}\det(W)=f\quad \text{on}\ \mathbb{S}^n.
\end{equation}

For a class of convex bodies that are close to the ball, by applying integral formulas involving the radial function, we obtain the uniqueness of solutions to $L_p$ dual Minkowski problems for $p=2$ and $q\le n+1$.

\begin{thm}\label{mainthm-5}
Let $n\ge 2$. Suppose $p=2$ and $q\le n+1$. For any positive function $f\in C(\mathbb{S}^n)$, then the smooth strictly convex solution $M$ to (\ref{an-p,q-M-eq}) satisfying $\frac{r}{u}\le \sqrt{2}$ is unique, except when $p=q$, in which case it is unique up to scaling.
\end{thm}

\begin{rem}
If a convex body $K$ satisfies $R B^{n+1}\subseteq K \subseteq \sqrt{2}R B^{n+1}$ for some $R>0$, where $B^{n+1}$ denotes the Euclidean unit ball in $\mathbb{R}^{n+1}$, then its boundary $M=\partial K$ satisfies $\frac{r}{u}\le \sqrt{2}$. Consequently, the uniqueness of the $L_2$ dual Minkowski problem with $q\le n+1$ in $\mathbb{R}^{n+1}$ is guaranteed for these convex bodies.
\end{rem}


$\ $

When $f\equiv 1$ and $q=n+1$, (\ref{an-p,q-M-eq}) reduces to the isotropic $L_p$ Minkowski problem:
\begin{equation}\label{is-p-M-eq}
    u^{1-p}\det(W)=1\quad \text{on}\ \mathbb{S}^n
\end{equation}
It has been established, due to the works of Lutwak \cite{Lut93}, Andrews \cite{And99}, Andrews-Guan-Ni \cite{AGN16}, Choi-Daskalopoulos \cite{CD16}, as well as Brendle-Choi-Daskalopoulos \cite{BCD17}, that for $p>-n-1$, the only solutions $M$ to (\ref{is-p-M-eq}) are origin-centered spheres, while for $p=-n-1$, the solutions $M$ are origin-centered ellipsoids. Moreover, Saroglou \cite{Sa22}  obtained uniqueness results for a class of isotropic Orlicz-Minkowski problems, including (\ref{is-p-M-eq}) for $p>-n-1$. By using a local version of the Brunn-Minkowski inequality, Ivaki and Milman \cite{IM23} provided a new proof of the uniqueness of solutions to (\ref{is-p-M-eq}) for $-n-1\le p\le -1$ and $p=0$.

As for the uniqueness of solutions to the isotropic $L_p$ dual Minkowski problem (\ref{p,q-M-eq}), it is known that the solution is unique when $q<p$ \cite{HZ18} and is unique up to scaling when $p=q$ \cite{CL21}, which can be obtained by the strong maximum principle. When $1<p<q\le n+1$ or $-n-1\le p<q<-1$, Chen-Li \cite{CL21} showed that the solution must be an origin-centred sphere. In the planar case $n=1$, we \cite{LW22} provided a complete classification of solutions to (\ref{p,q-M-eq}). When $-n-1\le p< q\le\min\{n+1,n+1+p\}$, Chen-Huang-Zhao \cite{CHZ19} proved that the origin-symmetric solution must be an origin-centred sphere. Furthermore, Ivaki and Milman \cite{IM23} proved the uniqueness of origin-symmetric solutions to (\ref{p,q-M-eq}) for $p\ge -n-1$ and $q\le n+1$ with at least one of these being strict.

In this paper, by selecting test functions in a local version of the Brunn-Minkowski inequality, we obtain some new uniqueness results for the isotropic $L_p$ dual Minkowski problem. 

\begin{thm}\label{mainthm-6}
Let $n\ge 1$. Suppose either
\begin{enumerate}[(i)]
	\item \label{thm-6i} $p\ge -n-1,\ q\le n-1+2\sqrt{1+\frac{n+1+p}{n+2}}$ and at least one of these is strict, or
	\item \label{thm-6ii} $q\le n+1,\ p\ge -n+1-2\sqrt{1+\frac{n+1-q}{n+2}}$ and at least one of these is strict.
\end{enumerate} 
Then the smooth strictly convex origin-symmetric solution $M$ to (\ref{p,q-M-eq}) must be an origin-centred sphere.
\end{thm}

\begin{rem}
In fact, the uniqueness of origin-symmetric solutions to the isotropic $L_p$ Christoffel-Minkowski problem (\ref{i-pq-CM}) was also investigated in \cite{IM23}. By modifying the test functions in \cite[Lemma 3.1]{IM23} for $k<n$, we obtain some results similar to Theorem \ref{mainthm-6}, see Appendix \ref{appendix}.
\end{rem}

When $p\ge -n-1$ and $q\le n+1$, Theorem \ref{mainthm-6} reduces to the following result of Ivaki and Milman.

\begin{cor}[\cite{IM23}]\label{cor-8}
Let $n\ge 1$. Suppose $p\ge -n-1$, $q\le n+1$ and at least one of these inequalities is strict. Then the smooth strictly convex origin-symmetric solution $M$ to the isotropic $L_p$ dual Minkowski problem (\ref{p,q-M-eq}) must be an origin-centred sphere.  
\end{cor}

In the planar case $n=1$, combining Theorem \ref{mainthm-6} and \cite[Theorem 1.2]{LW22} immediately yields the following corollary.
\begin{cor}\label{cor-9}
Let $n=1$. If $q>2$, $p+3<q\le 2\sqrt{\frac{5+p}{3}}$; or $p<-2$, $-2\sqrt{\frac{5-q}{3}}\le p < q-3$, then the embedded solution to the planar isotropic $L_p$ dual Minkowski problem is unique.
\end{cor}

As for the case without the origin-symmetric assumption, we have the following result.

\begin{thm}\label{mainthm-7}
Let $n\ge 1$. Suppose either
\begin{enumerate}[(i)]
	\item \label{thm-7i} $-n-1<p\le -1$, $n+1\le q\le n+\frac{1}{2}+\sqrt{\frac{1}{4}-\frac{(1+p)(n+1+p)}{n(n+2)}}$, or 
	\item \label{thm-7ii} $1\le q< n+1$, $-n-\frac{1}{2}-\sqrt{\frac{1}{4}-\frac{(1-q)(n+1-q)}{n(n+2)}}\le p\le -n-1$.
\end{enumerate} 
Then the smooth strictly convex solution $M$ to (\ref{p,q-M-eq}) must be an origin-centred sphere.  
\end{thm}

When $q=n+1$, Theorem \ref{mainthm-7} reduces to the following uniqueness result for the isotropic $L_p$ Minkowski problem (\ref{is-p-M-eq}) within the range $-n-1<p\le -1$, originally established by Brendle, Choi and Daskalopoulos \cite{BCD17}. 

\begin{cor}[\cites{BCD17,IM23,Sa22}]\label{cor-10}
Let $n\ge 1$. Suppose $-n-1< p\le -1$. Then the smooth strictly convex solution $M$ to the isotropic $L_p$ Minkowski problem (\ref{is-p-M-eq}) must be an origin-centred sphere.  
\end{cor}


$\ $

The paper is organized as follows. In Section \ref{sec:2}, we collect some basic concepts and techniques.
In Section \ref{sec:3}, we prove Theorem \ref{main-thm-1}, Theorem \ref{mainthm-2} and Theorem \ref{mainthm-5} by use of integral formulas. In Section \ref{sec:4}, inspired by Simon's result \cite{Si67}, we give the proof of Theorem \ref{mainthm-3}. In Section \ref{sec:5}, using the method developed by Ivaki and Milman \cite{IM23}, we prove Theorem \ref{mainthm-6} and Theorem \ref{mainthm-7}. In Appendix \ref{appendix}, we provide a uniqueness result of origin-symmetric solutions to the isotropic $L_p$ dual Christoffel-Minkowski problem.

\begin{ack}
	The work was supported by NSFC Grant No. 11831005.
\end{ack}

\section{Preliminaries}\label{sec:2}

\subsection{Inequalities of elementary symmetric functions}\label{sec:2.1}$\ $

We first review some properties of elementary symmetric functions, see e.g. \cites{CNS85, Ni00, Ga59}.

\begin{defn}
For an integer $k=1,\ldots,n$ and a point $\lambda=(\lambda_1,\ldots,\lambda_n)\in \mathbb{R}^n$, the \textit{normalized $k$-th elementary symmetric function} $P_k(\lambda)$ is defined by
\begin{align*}
P_k(\lambda)=\binom{n}{k}^{-1}\sigma_k(\lambda)=\binom{n}{k}^{-1}\sum_{1\leq i_1<\cdots<i_k \leq n}\lambda_{i_1}\cdots \lambda_{i_k}.
\end{align*}
It is convenient to set $P_0(\lambda)=1$ and $P_k(\lambda)=0$ for $k>n$.
\end{defn}

The definition can be extended to symmetric matrices. Let $A\in \mathrm{Sym}(n)$ be an $n\times n$ symmetric matrix and $\lambda=\lambda(A)$ be the eigenvalues of $A$. Set $P_k(A)=P_k(\lambda(A))$. We have
\begin{align*}
  P_k(A) =& \frac{(n-k)!}{n!}\delta_{i_1,\ldots,i_k}^{j_1,\ldots,j_k}A_{i_1j_1}\cdots A_{i_k j_k},\quad k=1,\ldots,n,
\end{align*}
where $\delta_{i_1,\ldots,i_k}^{j_1,\ldots,j_k}=\det(\delta_{i_s}^{j_t})_{k\times k}$ denotes the generalized Kronecker delta.

\begin{lem}\label{Lem 2.1-1}
If $\lambda\in \Gamma_k=\{\lambda\in \mathbb{R}^n:\ P_i(\lambda)>0,\ i=1,\ldots,k\}$, we have the following Newton-MacLaurin inequality
\begin{align}\label{2.1}
  P_k(\lambda) P_{l-1}(\lambda)\leq P_l(\lambda) P_{k-1}(\lambda),\quad 1\leq l<k.
\end{align}
Equality holds if and only if $\lambda_1=\ldots=\lambda_n$.
\end{lem}

Furthermore, we state the following properties which can be found in \cite[Theorem 2.1]{Ni00}.
\begin{lem}[\cite{Ni00}]\label{Lem 2.1-2}
Suppose that $\alpha,\beta\in (0,1)$ and $k,l\in\mathbb{N}$ are numbers such that
\begin{align*}
    \alpha+\beta=1\quad \text{and}\quad k\alpha+l\beta\in\{0,\ldots,n\}.
\end{align*} 

If $\lambda\in \Gamma_k\cap\Gamma_l$, we have 
\begin{align}\label{2.2}
    P_{k\alpha+l\beta}(\lambda)\ge P_k^{\alpha}(\lambda) P_l^{\beta}(\lambda).
\end{align}
Equality holds if and only if $\lambda_1=\ldots=\lambda_n$.    
\end{lem}

\begin{defn}
For an integer $k=1,\ldots,n$ and $k$ points $\lambda^{(1)},\ldots,\lambda^{(k)}\in \mathbb{R}^n$, the \textit{normalized $k$-th mixed elementary symmetric function} $P_k(\lambda^{(1)},\ldots,\lambda^{(k)})$ is defined by
\begin{align*}
    P_k(\lambda^{(1)},\ldots,\lambda^{(k)})=\binom{n}{k}^{-1}\sigma_k(\lambda^{(1)},\ldots,\lambda^{(k)})
    =\binom{n}{k}^{-1}\sum_{1\leq i_1<\cdots<i_k \leq n}\lambda_{i_1}^{(1)}\cdots \lambda_{i_k}^{(k)}.
\end{align*}  
\end{defn}

The definition can also be extended to symmetric matrices. Let $A^1,\ldots,A^k\in \mathrm{Sym}(n)$ be $n\times n$ symmetric matrices and $\lambda^{(i)}=\lambda(A^i)$ be the eigenvalues of $A^i$ for $i=1,\ldots,k$. Set $P_k(A^1,\ldots,A^k)=P_k(\lambda^{(1)},\ldots,\lambda^{(k)})$. We have
\begin{align*}
  P_k(A^1,\ldots,A^k)
  &=\frac{(n-k)!}{n!}\delta_{i_1,\ldots,i_k}^{j_1,\ldots,j_k}A_{i_1j_1}^1\cdots A_{i_k j_k}^k.
\end{align*}

It is obvious that 
\begin{align*}
    P_k(\lambda,\ldots,\lambda)=P_k(\lambda)\quad \text{and}\quad 
    P_n(\lambda^{(1)},\ldots,\lambda^{(k)},I,\ldots,I)=P_k(\lambda^{(1)},\ldots,\lambda^{(k)}),
\end{align*}
equivalently,
\begin{align*}
    P_k(A,\ldots,A)=P_k(A)\quad \text{and}\quad 
    P_n(A^1,\ldots,A^k,Id,\ldots,Id)=P_k(A^1,\ldots,A^k),
\end{align*}
where $I=(1,\ldots,1)\in\mathbb{R}^n$, and $Id$ is the $n\times n$ identity matrix.

Since $P_k$ is hyperbolic and complete, we have the following lemma, see \cite[Theorem 5]{Ga59}.
\begin{lem}[\cite{Ga59}]\label{Lem 2.1-3}
For any $A^i\in \Gamma_k,\ i=1,\ldots,k$, we have
\begin{align}\label{2.3}
    P_k(A^1,\ldots,A^k)\ge P_k(A^1)^{\frac{1}{k}}\cdots P_k(A^k)^{\frac{1}{k}}.
\end{align}
Equality holds if and only if these $k$ matrices are pairwise proportional. 
\end{lem}

In particular, define $P_{k,l}(\lambda,\Bar{\lambda})=P_{k+l}(\underbrace{\lambda,\ldots,\lambda}_{k},\underbrace{\Bar{\lambda},\ldots,\Bar{\lambda}}_{l})$. Then we have
\begin{lem}\label{Lem 2.1-4}
For any $\lambda,\Bar{\lambda}\in \Gamma_{k+l}$,
\begin{align}\label{2.4}
    P_{k,l}(\lambda,\Bar{\lambda})\ge P_{k+l}(\lambda)^{\frac{k}{k+l}} P_{k+l}(\Bar{\lambda})^{\frac{l}{k+l}},
\end{align}
Equality holds if and only if $\lambda$ and $\Bar{\lambda}$ are proportional. 
\end{lem}

\subsection{$L_k$ operator}\label{sec:2.2} $\ $

Let $x:M^n\to\mathbb{R}^{n+1}$ be a hypersurface, and choose a unit normal vector $e_{n+1}$ such that $\{e_1,\ldots,e_n,e_{n+1}\}$ is a local orthonormal basis. Let $(h_{ij})$ be the second fundamental form and $\sigma_k(h_{ij})$ be the $k$-th mean curvature of $M$. Denote by $u=-\langle x,e_{n+1}\rangle$ the support function and by $P_k(h_{ij})=\binom{n} {k}^{-1}\sigma_k(h_{ij})$ the normalized $k$-th mean curvature of $M$.

\begin{defn}
For a hypersurface $x:M^n\to\mathbb{R}^{n+1}$, the \textit{$k$-th Newton tensor} is defined by
\begin{align*}
   T_{ij}^k
   =\sigma_k(h_{ij}) \delta_{ij}-\sigma_{k-1}(h_{ij})h_{ij}+\ldots+(-1)^k \sum\limits_{i_1,\ldots,i_k} h_{i i_1}h_{i_1 i_2}\cdots h_{i_k j},\quad k=0,\ldots,n-1.
\end{align*}
We have a second-order differential operator $L_k$ defined by
\begin{align*}
    L_k f=\sum_{i,j}T_{ij}^k f_{ij},\quad \forall\ f\in C^{\infty}(M),\ 0\le k\le n-1.
\end{align*}
\end{defn}

$L_k$ operator satisfies the following properties, see \cites{Re73, HL08}:
\begin{prop}\label{prop-1}
\begin{enumerate}[(i)]
    \item  $T_{ij}^k=T_{ji}^k,\ \sum_j T_{ij,j}^k=0$ and $T_{ij}^k=\frac{1}{k!}\delta_{i_1,\ldots,i_k,i}^{j_1,\ldots,j_k,j}h_{i_1 j_1}\cdots h_{i_k j_k}$.

    \item $L_k$ operator is elliptic and self-adjoint for a strictly convex hypersurface $M$.

    \item  We have
    \begin{align}\label{2.5}
    & \frac{1}{2} L_k (|x|^2)=\sum T_{ij}^k (\langle x,e_i\rangle)_j
    =(n-k)\binom{n}{k}(P_k(h_{ij})-u P_{k+1}(h_{ij})).
\end{align}
\end{enumerate}
\end{prop}

\subsection{Integral formulas for hypersurfaces}\label{sec:2.3}$\ $

In this subsection, we review some integral formulas derived by S. S. Chern \cite{Ch59} and provide some new integral formulas that contain radial functions.

Let $x:M^n\to\mathbb{R}^{n+1}$ and $\Bar{x}:\Bar{M}^n\to\mathbb{R}^{n+1}$ be two closed strictly convex smooth hypersurfaces such that the normal vectors of $M$ at $x$ and of $\Bar{M}$ at $\Bar{x}$ are the same. Let $\{e_1,\ldots, e_n,e_{n+1}\}$ be an orthonormal frame in $\mathbb{R}^{n+1}$ such that $e_{n+1}$ is the unit inner normal vector of $M$ at $x$, and also the unit inner normal vector of $\Bar{M}$ at $\Bar{x}$. We will denote the following ranges of indices:
$$1\le A,B,\ldots\le n+1;\quad 1\le i,j,\ldots\le n.$$

Let $\{\theta_A\}$ be the dual frame field of $\{e_A\}$, and denote $\omega_A=\left.\theta_A\right|_M$ and $\Bar{\omega}_A=\left.\theta_A\right|_{\Bar{M}}$. Restricting $\theta_A$ to the hypersurface $M$, we have $\theta_i=\omega_i$ and $\theta_{n+1}=0$. We have
$$dx=\sum\limits_{i}\omega_i e_i,\quad de_{n+1}=\sum\limits_{i}\omega_{n+1,i}e_i,$$
where $\omega_{i,n+1}=\sum\limits_{j}h_{ij}\omega_j$. Let $x=\sum\limits_{A}\tau_A e_A$, $(h^{ik})=(h_{ik})^{-1}$, and denote by $u=-\langle x,e_{n+1}\rangle$ the support function and by $r=\langle x,x\rangle^{\frac{1}{2}}$ the radial function of $M$, then we obtain
\begin{align}\label{2.6}
    \tau_{n+1}=-u,\quad r^2=\sum\limits_{i}\tau_i^2+u^2,\quad 
    dr=\frac{1}{r}\sum\limits_{i}\tau_i\omega_i.
\end{align}
We assign notations $\Bar{\omega}_{n+1,i},\ \Bar{h}_{ij},\ \Bar{h}^{ij},\ \Bar{\tau}_A,\ \Bar{u}$ and $\Bar{r}$ for hypersurface $\Bar{M}$ as well.

Notice that $\omega_{n+1,i}=\langle de_{n+1},e_i\rangle=\Bar{\omega}_{n+1,i}$. Let $a_{ij}=\sum\limits_k h^{ik}\Bar{h}_{kj}$, we have
\begin{align}\label{2.7}
    \omega_i=\sum\limits_j a_{ij}\Bar{\omega}_j.
\end{align}

Consider the differential form of degree $n-1$:
\begin{align}\label{2.8}
    \Phi_0=(x,\Bar{x},\underbrace{d\Bar{x},\ldots,d\Bar{x}}_{n-1}),
\end{align}
where $(\cdot,\ldots,\cdot)$ is a determinant of order $n+1$, whose columns are the components of the respective vectors or vector-valued differential forms, with the convention that in the expansion of the determinant the multiplication of differential forms is in the sense of exterior multiplication. Then applying exterior differentiation gives
\begin{equation}\label{2.9}
\begin{split}
    d\Phi_0 &=(x,d\Bar{x},\underbrace{d\Bar{x},\ldots,d\Bar{x}}_{n-1})+(dx,\Bar{x},\underbrace{d\Bar{x},\ldots,d\Bar{x}}_{n-1})\\
    &=(\sum_i \tau_i e_i-u e_{n+1},\sum_i\Bar{\omega}_i e_i,\underbrace{\sum_i\Bar{\omega}_i e_i,\ldots,\sum_i\Bar{\omega}_i e_i}_{n-1})\\
    &\quad +(\sum_i \omega_i e_i,\sum_i \Bar{\tau}_i e_i-\Bar{u} e_{n+1},\underbrace{\sum_i\Bar{\omega}_i e_i,\ldots,\sum_i\Bar{\omega}_i e_i}_{n-1})\\
    &=-(-1)^n u\sum \delta_{i_1,\ldots,i_n}^{1,\ldots,n}\Bar{\omega}_{i_1}\wedge\ldots\wedge \Bar{\omega}_{i_n}
    +(-1)^n \Bar{u}\sum \delta_{i_1,\ldots,i_n}^{1,\ldots,n}\omega_{i_1}\wedge\Bar{\omega}_{i_2}\wedge\ldots\wedge \Bar{\omega}_{i_n}\\
    &=(-1)^n n!(-u+\Bar{u}P_1(a_{ij}))d\Bar{\mu},
\end{split}
\end{equation}
where $d\Bar{\mu}=\Bar{\omega}_1\wedge\ldots\wedge\Bar{\omega}_n$. Since $M$ and $\Bar{M}$ are closed and strictly convex, we can identify $M,\Bar{M}$ with $\mathbb{S}^n$ by use of Gauss maps. Then previously mentioned functions and differential forms are now defined on $\mathbb{S}^n$. Therefore, integrating (\ref{2.9}) over $\mathbb{S}^n$ yields
\begin{align}\label{2.10}
    \int_{\mathbb{S}^n}(\Bar{u}P_1(a_{ij})-u)d\Bar{\mu}=0.
\end{align}

Similarly, we consider a family of differential forms as follows:
\begin{align}\label{2.11}
    \Phi_{k}=(x,\Bar{x},\underbrace{dx,\ldots,dx}_{k},\underbrace{d\Bar{x},\ldots,d\Bar{x}}_{n-k-1}),\quad k=1,\ldots,n-1.
\end{align}
Applying exterior differentiation gives
\begin{align}\label{2.12}
    d\Phi_k=(-1)^n n!(\Bar{u}P_{k+1}(a_{ij})-u P_{k}(a_{ij}))d\Bar{\mu}.
\end{align}
Then integrating (\ref{2.12}) over $\mathbb{S}^n$ yields
\begin{align}\label{2.13}
    \int_{\mathbb{S}^n}(\Bar{u}P_{k+1}(a_{ij})-u P_{k}(a_{ij}))d\Bar{\mu}=0.
\end{align}

In conclusion, we obtain the following integral formulas:
\begin{lem}[\cite{Ch59}]\label{Lem 2.3-1}
Let $x:M^n\to\mathbb{R}^{n+1}$ and $\Bar{x}:\Bar{M}^n\to\mathbb{R}^{n+1}$ be two closed smooth strictly convex hypersurfaces. Then, for $k=0,\ldots,n-1$
\begin{align}
    &\int_{\mathbb{S}^n}(\Bar{u}P_{k+1}(a_{ij})-u P_{k}(a_{ij}))d\Bar{\mu}=0, \label{2.14}\\
    &\int_{\mathbb{S}^n}(u P_{k+1}(a^{ij})-\Bar{u} P_{k}(a^{ij}))d\mu=0, \label{2.15}
\end{align}
where $a_{ij}=\sum\limits_s h^{is}\Bar{h}_{sj}$ and $a^{ij}=\sum\limits_s \Bar{h}^{is}h_{sj}$.
\end{lem}

In particular, when $\Bar{M}$ is the unit sphere $\mathbb{S}^n$, we have $\Bar{u}=1$ and $(a^{ij})=(h_{ij})$. Note that the eigenvalues of $(h_{ij})$ are the principal curvatures $\{\kappa_i\}$ of $M$, then $P_k(a^{ij})$ is the normalized $k$-th mean curvature $P_k(\kappa)$ of $M$. Hence, (\ref{2.15}) reduces to Minkowski formulas:
\begin{align}\label{2.16}
    \int_{M}(u P_{k+1}(\kappa)-P_{k}(\kappa))d\mu=0,\quad k=0,\ldots,n-1.
\end{align}

$\ $

Next, we consider the following differential forms:
\begin{align}\label{2.17}
     \Phi_{k,l}=(x,\Bar{x},\underbrace{dx,\ldots,dx}_{k},\underbrace{de_{n+1},\ldots,de_{n+1}}_{n-k-l-1},\underbrace{d\Bar{x},\ldots,d\Bar{x}}_{l}),\quad 0\le k+l\le n-1.
\end{align}
Denote by $d\sigma$ the standard area element of $\mathbb{S}^n$, it follows that $d\sigma=\omega_{1,n+1}\wedge\ldots\wedge\omega_{n,n+1}$, and then
$$d\mu=\det(h^{ij})d\sigma,\quad d\Bar{\mu}=\det(\Bar{h}^{ij})d\sigma.$$
Applying exterior differentiation to (\ref{2.17}) gives
\begin{equation}\label{2.18}
\begin{split}
    d\Phi_{k,l} 
    &=-(-1)^{k+l+1} (n-k-l-1)! u\sum\delta_{i_1,\ldots i_{k+l+1}}^{j_1,\ldots j_{k+l+1}} h^{i_1 j_1}\cdots h^{i_k j_k}\Bar{h}^{i_{k+1}j_{k+1}}\cdots\Bar{h}^{i_{k+l+1}j_{k+l+1}}d\sigma\\
    &\quad +(-1)^{k+l+1} (n-k-l-1)!\Bar{u}\sum\delta_{i_1,\ldots i_{k+l+1}}^{j_1,\ldots j_{k+l+1}} h^{i_1 j_1}\cdots h^{i_{k+1} j_{k+1}}\Bar{h}^{i_{k+2}j_{k+2}}\cdots\Bar{h}^{i_{k+l+1}j_{k+l+1}}d\sigma\\
    &=(-1)^{k+l+1} n! \left(\Bar{u}P_{k+1,l}(h^{ij},\Bar{h}^{ij})-u P_{k,l+1}(h^{ij},\Bar{h}^{ij})\right)d\sigma.
\end{split}  
\end{equation}
Hence, integrating (\ref{2.18}) over $\mathbb{S}^n$ yields

\begin{lem}[\cite{Ch59}]\label{Lem 2.3-2}
Let $x:M^n\to\mathbb{R}^{n+1}$ and $\Bar{x}:\Bar{M}^n\to\mathbb{R}^{n+1}$ be two closed smooth strictly convex hypersurfaces. Then, for $0\le k+l\le n-1$
\begin{align}\label{2.19}
    \int_{\mathbb{S}^n}\left(\Bar{u}P_{k+1,l}(h^{ij},\Bar{h}^{ij})-u P_{k,l+1}(h^{ij},\Bar{h}^{ij})\right)d\sigma=0.  
\end{align}
\end{lem}

Generally, by applying the approach in \cite{Gu10}, we can generalize integral formulas (\ref{2.19}) to the cases where $u$ and $\Bar{u}$ are $C^2$ functions.
\begin{lem}\label{Lem 2.3-3}
Let $u$ and $\Bar{u}$ be two $C^2$ functions on $\mathbb{S}^n$. Denote $(h^{ij})=(u_{ij}+u\delta_{ij})$ and $(\Bar{h}^{ij})=(\Bar{u}_{ij}+\Bar{u}\delta_{ij})$. Then, for $0\le k+l\le n-1$
\begin{align}\label{2.20}
    \int_{\mathbb{S}^n}\left(\Bar{u}P_{k,l}(h^{ij},\Bar{h}^{ij})- P_{k,l+1}(h^{ij},\Bar{h}^{ij})\right)d\sigma=0,  
\end{align}
\begin{align}\label{2.21}
    \int_{\mathbb{S}^n}\left(u P_{k,l}(h^{ij},\Bar{h}^{ij})-P_{k+1,l}(h^{ij},\Bar{h}^{ij})\right)d\sigma=0,  
\end{align}
and
\begin{align}\label{2.22}
    \int_{\mathbb{S}^n}\left(\Bar{u}P_{k+1,l}(h^{ij},\Bar{h}^{ij})-u P_{k,l+1}(h^{ij},\Bar{h}^{ij})\right)d\sigma=0.  
\end{align}
\end{lem}

$\ $

Finally, we derive some new integral formulas that contain radial functions.

Let $\eta:(0,+\infty)\to(0,+\infty)$ be a $C^1$ function. For any $k=1,\ldots,n-1$, we calculate the exterior differentiation of $\eta(\Bar{r})\Phi_k$:
\begin{align*}
    d(\eta(\Bar{r})\Phi_k)=\eta(\Bar{r})d\Phi_k+\frac{\eta'(\Bar{r})}{\Bar{r}}\sum_i\Bar{\tau}_i\Bar{\omega}_i\wedge \Phi_k.
\end{align*}
Using (\ref{2.6}), (\ref{2.7}) and (\ref{2.11}), we have
\begin{align*}
    \Phi_k 
    &=(\sum_i \tau_i e_i-u e_{n+1},\sum_i \Bar{\tau}_i e_i-\Bar{u} e_{n+1},\sum_i\Bar{\omega}_i e_i,\underbrace{\sum_i\omega_i e_i,\ldots,\sum_i\omega_i e_i}_{k},\underbrace{\sum_i\Bar{\omega}_i e_i,\ldots,\sum_i\Bar{\omega}_i e_i}_{n-k-1})\\
    &=(-1)^n \sum \delta_{i_1,\ldots,i_n}^{1,\ldots,n}(-u\Bar{\tau}_{i_1}+\Bar{u}\tau_{i_1})\underbrace{\omega_{i_2}\wedge\ldots\wedge\omega_{i_{k+1}}}_{k}\wedge\underbrace{\Bar{\omega}_{i_{k+2}}\wedge\ldots\wedge\Bar{\omega}_{i_n}}_{n-k-1} \\
    &=(-1)^n \sum \delta_{i_1,\ldots,i_n}^{1,\ldots,n}(-u\Bar{\tau}_{i_1}+\Bar{u}\tau_{i_1}) a_{i_2 j_2}\cdots a_{i_{k+1} j_{k+1}} \underbrace{\Bar{\omega}_{j_{2}}\wedge\ldots\wedge\Bar{\omega}_{j_{k+1}}}_{k}\underbrace{\Bar{\omega}_{i_{k+2}}\wedge\ldots\wedge\Bar{\omega}_{i_n}}_{n-k-1}\\
    &=(-1)^n \sum \delta_{i_1,j_2\ldots,j_{k+1},i_{k+2},\ldots,i_n}^{1,2,\ldots,k+1,k+2,\ldots,n}(\Bar{u}\tau_{i_1}-u\Bar{\tau}_{i_1}) a_{i_2 j_2}\cdots a_{i_{k+1} j_{k+1}}\underbrace{\Bar{\omega}_{i_{2}}\wedge\ldots\wedge\Bar{\omega}_{i_n}}_{n-1}.
\end{align*}
Combining this with (\ref{2.12}), we obtain
\begin{equation}\label{2.23}
\begin{split}
    d(\eta(\Bar{r})\Phi_k)
    &=(-1)^n n!\eta(\Bar{r})(\Bar{u}P_{k+1}(a_{ij})-u P_{k}(a_{ij}))d\Bar{\mu}\\
    &\quad +(-1)^n\frac{\eta'(\Bar{r})}{\Bar{r}} \sum \delta_{i_1,j_2\ldots,j_{k+1},i_{k+2},\ldots,i_n}^{1,2,\ldots,{k+1},k+2,\ldots,n} \Bar{\tau}_i(\Bar{u}\tau_{i_1}-u\Bar{\tau}_{i_1}) a_{i_2 j_2}\cdots a_{i_{k+1} j_{k+1}}\Bar{\omega}_i\wedge\underbrace{\Bar{\omega}_{i_{2}}\wedge\ldots\wedge\Bar{\omega}_{i_n}}_{n-1}\\
    &=(-1)^n n!\eta(\Bar{r})(\Bar{u}P_{k+1}(a_{ij})-u P_{k}(a_{ij}))d\Bar{\mu}\\
    &\quad +(-1)^n (n-k-1)!\frac{\eta'(\Bar{r})}{\Bar{r}}\sum \delta_{i_1,i_2\ldots,i_{k+1}}^{i_1,j_2,\ldots,j_{k+1}} (\Bar{u}\Bar{\tau}_{i_1}\tau_{i_1}-u\Bar{\tau}_{i_1}^2) a_{i_2 j_2}\cdots a_{i_{k+1} j_{k+1}} d\Bar{\mu}.
\end{split}
\end{equation}
When $k=0$, applying exterior differentiation of $\eta(\Bar{r})\Phi_0$ gives
\begin{equation}\label{2.24}
    d(\eta(\Bar{r})\Phi_0)
    =(-1)^n n!\eta(\Bar{r})(\Bar{u}P_1(a_{ij})-u)d\Bar{\mu}
    +(-1)^n (n-1)!\frac{\eta'(\Bar{r})}{\Bar{r}}\sum_i(\Bar{u}\Bar{\tau}_{i_1}\tau_{i_1}-u\Bar{\tau}_{i_1}^2)d\Bar{\mu}.
\end{equation}

Therefore, integrating (\ref{2.23}) and (\ref{2.24}) over $\mathbb{S}^n$ yields
\begin{lem}\label{Lem 2.3-4}
Let $x:M^n\to\mathbb{R}^{n+1}$ and $\Bar{x}:\Bar{M}^n\to\mathbb{R}^{n+1}$ be two closed smooth strictly convex hypersurfaces. Then, for $k=0,\ldots,n-1$
\begin{equation}\label{2.25}
\begin{split}
     &\int_{\mathbb{S}^n}\eta(\Bar{r})(\Bar{u}P_{k+1}(a_{ij})-u P_{k}(a_{ij}))d\Bar{\mu}\\
    =&\frac{(n-k-1)!}{n!}\int_{\mathbb{S}^n}\frac{\eta'(\Bar{r})}{\Bar{r}}\sum \delta_{i_1,i_2\ldots,i_{k+1}}^{i_1,j_2,\ldots,j_{k+1}} (u\Bar{\tau}_{i_1}^2-\Bar{u}\Bar{\tau}_{i_1}\tau_{i_1}) a_{i_2 j_2}\cdots a_{i_{k+1} j_{k+1}} d\Bar{\mu},
\end{split}
\end{equation}
and
\begin{equation}\label{2.26}
\begin{split}
     &\int_{\mathbb{S}^n}\eta(r)(u P_{k+1}(a^{ij})-\Bar{u} P_{k}(a^{ij}))d\mu\\
    =&\frac{(n-k-1)!}{n!}\int_{\mathbb{S}^n}\frac{\eta'(r)}{r}\sum \delta_{i_1,i_2\ldots,i_{k+1}}^{i_1,j_2,\ldots,j_{k+1}} (\Bar{u}\tau_{i_1}^2-u \tau_{i_1}\Bar{\tau}_{i_1}) a^{i_2 j_2}\cdots a^{i_{k+1} j_{k+1}} d\mu,
\end{split}
\end{equation}
where the terms $ a_{i_2 j_2}\cdots a_{i_{k+1} j_{k+1}}$ and $a^{i_2 j_2}\cdots a^{i_{k+1} j_{k+1}}$ on the right hand sides vanish when $k=0$.
\end{lem}

Moreover, when $\Bar{M}$ is the unit sphere $\mathbb{S}^n$, then $\Bar{r}=\Bar{u}=1$, $\Bar{\tau}_i=0$ and $(a^{ij})=(h_{ij})$. Using (\ref{2.26}), we obtain
\begin{lem}\label{Lem 2.3-5}
Let $x:M^n\to\mathbb{R}^{n+1}$ be a closed smooth strictly convex hypersurface. Then, for $k=0,\ldots,n-1$
\begin{equation}\label{2.27}
\begin{split}
     &\int_{\mathbb{S}^n}\eta(r)(u P_{k+1}(h_{ij})- P_{k}(h_{ij}))d\mu\\
    =&\frac{(n-k-1)!}{n!}\int_{\mathbb{S}^n}\frac{\eta'(r)}{r}\sum \delta_{i_1,i_2\ldots,i_{k+1}}^{i_1,j_2,\ldots,j_{k+1}} \tau_{i_1}^2 h_{i_2 j_2}\cdots h_{i_{k+1} j_{k+1}} d\mu,
\end{split}
\end{equation}
where the term $h_{i_2 j_2}\cdots h_{i_{k+1} j_{k+1}}$ on the right hand side vanishes when $k=0$.
\end{lem}

Note that when $\eta\equiv1$, Lemma \ref{Lem 2.3-4} reduces to Lemma \ref{Lem 2.3-1}, and Lemma \ref{Lem 2.3-5} reduces to Minkowski formulas (\ref{2.16}).

\subsection{Spectral estimate}\label{sec:2.4}$\ $

In this subsection, we will use the same notation as in subsection \ref{sec:2.3}. Let $(\mathbb{S}^n,g_0,\nabla)$ denote the unit sphere equipped with its standard round metric and Levi-Civita connection. For a smooth closed strictly convex hypersurface $x:M^n\to\mathbb{R}^{n+1}$, denote by $\nu:M^n\to\mathbb{S}^n$ the Gauss map of $M$. For convenience, we denote $\sigma_k=\sigma_k(h^{ij})$ for $1\le k\le n$ and $dV_n=u\sigma_n d\sigma$. 

The following lemma is the local version of the Alexandrov-Fenchel inequality and can also be regarded as a spectral interpretation of the Brunn-Minkowski inequality. For more details, refer to \cites{An97,ACGL20,KM22,Mil21,IM23}.
\begin{lem}[\cite{An97,ACGL20}]\label{Lem 2.4-1}
Let $f\in C^2(\mathbb{S}^n)$ with $\int_{\mathbb{S}^n}f u\sigma_n d\sigma=0$. Then we have
\begin{equation}\label{2.28}
    n\int_{\mathbb{S}^n}f^2 u\sigma_n d\sigma\le \int_{\mathbb{S}^n}u^2 \sigma_n^{ij}\nabla_i f\nabla_j f d\sigma.
\end{equation}
Equality holds if and only if for some vector $v\in\mathbb{R}^{n+1}$ we have
\begin{equation*}
    f(z)=\langle \frac{z}{u(z)},v\rangle,\quad\forall\ z\in\mathbb{S}^n.
\end{equation*}
\end{lem}

It is well-known (see e.g. \cite{Sch14}) that the inverse Gauss map $X=\nu^{-1}:\mathbb{S}^n\to M$ is given by
\begin{align*}
    X(z)=u(z)z+\nabla u(z),\quad \forall\ z\in\mathbb{S}^n.
\end{align*}
By substituting some test functions in (\ref{2.28}), we obtain a key inequality as follows:
\begin{lem}\label{Lem 2.4-2}
Assume that $\alpha\in\mathbb{R}$. Then we have
\begin{equation}\label{2.29} 
\begin{split}
    n\int_{\mathbb{S}^n}|X|^{2\alpha+2}dV_n
    & \le 
    n\frac{|\int_{\mathbb{S}^n}|X|^{\alpha}X dV_n|^2}{\int_{\mathbb{S}^n}dV_n}
    +\int_{\mathbb{S}^n}|X|^{2\alpha}u(\Delta u+n u) dV_n\\
    &\quad +(\alpha^2+2\alpha)\int_{\mathbb{S}^n}|X|^{2\alpha-1}u \langle \nabla u,\nabla |X|\rangle dV_n.
\end{split}
\end{equation}
\end{lem}

\begin{proof}
Let $\{E_l\}_{l=1}^{n+1}$ be an orthonormal basis of $\mathbb{R}^{n+1}$. Suppose $\{e_i\}_{i=1}^n$ is a local orthonormal frame for $\mathbb{S}^n$ such that $(u_{ij}+u\delta_{ij})(z_0)=\lambda_i(z_0)\delta_{ij}$. Inspired by the construction of Ivaki and Milman \cite{IM23}, for $l=1,\ldots,n+1$, we define the functions $f_l:\mathbb{S}^n\to\mathbb{R}$
\begin{equation}\label{2.30}
    f_l(z)=|X(z)|^{\alpha}\langle X(z),E_l\rangle-\frac{\int_{\mathbb{S}^n}|X(z)|^{\alpha}\langle X(z),E_l\rangle dV_n}{\int_{\mathbb{S}^n}dV_n}.
\end{equation}
Since $\int_{\mathbb{S}^n}f_l dV_n=0$ for $1\le l\le n+1$, by applying Lemma \ref{Lem 2.4-1} to $f_l$ and summing over $l$, we have
\begin{align*}
     n\sum_l \int_{\mathbb{S}^n}f_l^2 dV_n
    =n\left[\int_{\mathbb{S}^n}|X|^{2\alpha+2}dV_n-\frac{|\int_{\mathbb{S}^n}|X|^{\alpha}X dV_n|^2}{\int_{\mathbb{S}^n}dV_n}\right]
    \le \sum_l \int_{\mathbb{S}^n}u^2 \sigma_n^{ij}\nabla_i f_l\nabla_j f_l d\sigma.
\end{align*}
It follows from  that $\nabla_i X=\sum_j (u_{ij}+u\delta_{ij})e_j=\lambda_i e_i$ at $z_0$. Then $\langle e_i,X\rangle=u_i$ and $\lambda_i\langle e_i,X\rangle^2=|X|\langle \nabla u,\nabla |X|\rangle$ at $z_0$. Using $\frac{\partial \sigma_n}{\partial \lambda_i}=\frac{\sigma_n}{\lambda_i}$ and $\sum_i \frac{\partial \sigma_n}{\partial \lambda_i}\lambda_i^2=\sigma_1\sigma_n$, we have
\begin{align*}
    \sum_{l,i,j}\sigma_n^{ij}\nabla_i f_l\nabla_j f_l
    &=\sum_{l,i}\frac{\partial \sigma_n}{\partial \lambda_i}(\nabla_i(|X|^{\alpha})\langle X,E_l\rangle+|X|^{\alpha}\langle \nabla_i X,E_l\rangle)^2\\
    &=\sum_{l,i}\frac{\partial \sigma_n}{\partial \lambda_i}(\alpha|X|^{\alpha-2}\langle \lambda_i e_i,X\rangle\langle X,E_l\rangle+|X|^{\alpha}\langle \lambda_i e_i,E_l\rangle)^2\\
    &=\sum_i \frac{\partial \sigma_n}{\partial \lambda_i}\lambda_i^2 (|X|^{2\alpha}+(\alpha^2+2\alpha)|X|^{2\alpha-2}\langle e_i,X\rangle^2)\\
    &= |X|^{2\alpha} \sigma_1\sigma_n+(\alpha^2+2\alpha)|X|^{2\alpha-1}\langle \nabla u,\nabla |X|\rangle\sigma_n.
\end{align*}
Therefore, we obtain
\begin{align*}
    & n\left[\int_{\mathbb{S}^n}|X|^{2\alpha+2}dV_n-\frac{|\int_{\mathbb{S}^n}|X|^{\alpha}X dV_n|^2}{\int_{\mathbb{S}^n}dV_n}\right]\\
    \le &\int_{\mathbb{S}^n}u^2 \left(|X|^{2\alpha} \sigma_1+(\alpha^2+2\alpha)|X|^{2\alpha-1}\langle \nabla u,\nabla |X|\rangle\right)\sigma_n d\sigma\\
    = &\int_{\mathbb{S}^n}u \left(|X|^{2\alpha} (\Delta u+n u)+(\alpha^2+2\alpha)|X|^{2\alpha-1}\langle \nabla u,\nabla |X|\rangle\right) dV_n.
\end{align*}
This completes the proof of Lemma \ref{Lem 2.4-2}.
\end{proof}

\begin{rem}
In the case $\alpha=0$, Lemma \ref{Lem 2.4-2} reduces to \cite[Lemma 3.2]{IM23} for $k=n$.
\end{rem}

\section{Proofs of Theorem \ref{main-thm-1}, Theorem \ref{mainthm-2} and Theorem \ref{mainthm-5}}\label{sec:3}

\begin{proof}[Proof of Theorem \ref{main-thm-1}]
Suppose there exist two admissible solutions $M$ and $\Bar{M}$ to (\ref{OCM-eq}), i.e.
\begin{align}\label{3.1}
    \varphi(u) P_k(W)=\varphi(\Bar{u}) P_k(\Bar{W})=f\quad \text{on}\ \mathbb{S}^n,
\end{align}
where $W=(\nabla^2 u +u g_0)$ and $\Bar{W}=(\nabla^2 \Bar{u}+\Bar{u} g_0).$ 

Denote $P_{k,l}=P_{k,l}(W,\Bar{W})$. Since $W,\Bar{W}\in\Gamma_k$, it follows from (\ref{3.1}) and Lemma \ref{Lem 2.1-4} that
\begin{align*}
    P_{k-1,1}\ge P_{k0}^{\frac{k-1}{k}}P_{0k}^{\frac{1}{k}}
    =\left(\frac{\varphi(\Bar{u})}{\varphi(u)}\right)^{\frac{k-1}{k}}P_{0k},\quad
    P_{1,k-1}\ge P_{k0}^{\frac{1}{k}}P_{0k}^{\frac{k-1}{k}}
    =\left(\frac{\varphi(u)}{\varphi(\Bar{u})}\right)^{\frac{k-1}{k}}P_{k0}.
\end{align*}
For $2\le k\le n$, by using integral formulas (\ref{2.22}) and the assumption (\ref{1.11}), we get
\begin{equation}\label{3.2}
    \begin{aligned}
    0&=\int_{\mathbb{S}^n} u(P_{0k}-P_{k-1,1})d\sigma+\int_{\mathbb{S}^n} \Bar{u}(P_{k0}-P_{1,k-1})d\sigma\\
    &\le \int_{\mathbb{S}^n} \left[u\left(1-\left(\frac{\varphi(\Bar{u})}{\varphi(u)}\right)^{\frac{k-1}{k}}\right)P_{0k}+ \Bar{u}\left(1-\left(\frac{\varphi(u)}{\varphi(\Bar{u})}\right)^{\frac{k-1}{k}}\right)P_{k0}\right]d\sigma\\
    &= \int_{\mathbb{S}^n}\left(1-\left(\frac{\varphi(\Bar{u})}{\varphi(u)}\right)^{\frac{k-1}{k}}\right)\left(1-\frac{\Bar{u}}{u}\left(\frac{\varphi(\Bar{u})}{\varphi(u)}\right)^{\frac{1}{k}}\right) u P_{0k}d\sigma\\
    &\le 0.
    \end{aligned}
\end{equation}
Hence the inequalities in (\ref{3.2}) are actually equalities. It follows that $W,\Bar{W}$ are pairwise proportional and $u=\Bar{u}$. Therefore, we obtain $M=\Bar{M}$.
\end{proof}

When $\varphi$ is $C^1$-smooth, we can reduce the condition (\ref{1.11}) to a characterization of $(\log\varphi)'$, and then obtain the following corollary.

\begin{cor}\label{cor-1}
Let $2\le k\le n$ be an integer. Let $\varphi:(0,\infty)\to(0,\infty)$ be a $C^1$ function that satisfies
\begin{equation}\label{cond-1}
\begin{split}
    & -\frac{k}{s}<(\log\varphi)'(s)<0,
\end{split}
\end{equation}
for any $s>0$. For any positive function $f\in C(\mathbb{S}^n)$, then the smooth admissible solution to (\ref{OCM-eq}) is unique. 
\end{cor}

\begin{rem}
Applying the maximum principle, it is clear that the conclusion of Corollary \ref{cor-1} still holds when the condition (\ref{cond-1}) is replaced by the condition $(\log\varphi)'(s)<-\frac{k}{s}$.
\end{rem}

$\ $

\begin{proof}[Proof of Theorem \ref{mainthm-5}]
Suppose there exist two smooth strictly convex solutions $M$ and $\Bar{M}$ to (\ref{an-p,q-M-eq}). Denote $a_{ij}=\sum_k h^{ik}\Bar{h}_{kj}$ and $P_{k,n-k}=P_{k,n-k}(W,\Bar{W})$, then we have
\begin{align}\label{3.3}
    \det(a_{ij})=\frac{P_{n,0}}{P_{0,n}}=\frac{\Bar{u}^{1-p}\Bar{r}^{q-n-1}}{u^{1-p} r^{q-n-1}}.
\end{align}

Using the integral formulas (\ref{2.25}) and (\ref{2.26}) with $k=0$, we obtain
\begin{align*}
     & \int_{\mathbb{S}^n} \left[\eta(\Bar{r})(\Bar{u} P_{1,n-1}-u P_{0,n})+\eta(r)(u P_{n-1,1}-\Bar{u} P_{n,0})\right]d\sigma\\
    =& \frac{1}{n}\int_{\mathbb{S}^n}\left[\frac{\eta'(\Bar{r})}{\Bar{r}}\sum_i (u\Bar{\tau}_{i}^2-\Bar{u}\Bar{\tau}_{i}\tau_{i})P_{0,n}
    +\frac{\eta'(r)}{r}\sum_i (\Bar{u}\tau_{i}^2-u\tau_{i}\Bar{\tau}_{i}) P_{n,0}\right] d\sigma\\
    \le & \frac{1}{n}\int_{\mathbb{S}^n}\left[\frac{\eta'(\Bar{r})}{\Bar{r}}\left(u (\Bar{r}^2-\Bar{u}^2)-\Bar{u}\sqrt{\Bar{r}^2-\Bar{u}^2}\sqrt{r^2-u^2}\right) P_{0,n}\right.\\
    & \quad \left. +\frac{\eta'(r)}{r}\left(\Bar{u} (r^2-u^2)-u\sqrt{r^2-u^2}\sqrt{\Bar{r}^2-\Bar{u}^2}\right) P_{n,0}\right] d\sigma\\
    =&\frac{1}{n}\int_{\mathbb{S}^n}\left(\sqrt{\left(\frac{\Bar{r}}{\Bar{u}}\right)^2-1}-\sqrt{\left(\frac{r}{u}\right)^2-1}\right)
    \left(\frac{\eta'(\Bar{r})}{\Bar{r}}\sqrt{\Bar{r}^2-\Bar{u}^2} P_{0,n}-\frac{\eta'(r)}{r}\sqrt{r^2-u^2} P_{n,0}\right)u\Bar{u} d\sigma.
\end{align*}
where we used the Cauchy-Schwarz inequality. It follows from Lemma \ref{Lem 2.1-4} that
\begin{align*}
    & \int_{\mathbb{S}^n} \left[\eta(\Bar{r})(\Bar{u} P_{1,n-1}-u P_{0,n})+\eta(r)(u P_{n-1,1}-\Bar{u} P_{n,0})\right]d\sigma\\
    \ge & \int_{\mathbb{S}^n} \left[\eta(\Bar{r})(\Bar{u} P_{0,n}^{\frac{n-1}{n}}P_{n,0}^{\frac{1}{n}}-u P_{0,n})+\eta(r)(u P_{n,0}^{\frac{n-1}{n}}P_{0,n}^{\frac{1}{n}}-\Bar{u} P_{n,0})\right]d\sigma\\
    =& \int_{\mathbb{S}^n} 
    \left(\frac{\Bar{u}}{u}\left(\frac{P_{n,0}}{P_{0,n}}\right)^{\frac{1}{n}}-1\right)
    \left(1-\frac{\eta(r)}{\eta(\Bar{r})}\left(\frac{P_{n,0}}{P_{0,n}}\right)^{\frac{n-1}{n}}\right)
    u\eta(\Bar{r}) P_{0,n} d\sigma
\end{align*}
Let $\eta(t)=t^{\alpha}$, where $\alpha$ is a non-positive constant to be determined. By (\ref{3.3}), we obtain
\begin{equation}\label{3.4}
\begin{split}
    & \int_{\mathbb{S}^n} 
    \left( \left(\frac{\Bar{u}}{u}\right)^{\frac{n+1-p}{n}}\left(\frac{\Bar{r}}{r}\right)^{\frac{q-n-1}{n}}-1\right)
    \left(1-\left(\frac{\Bar{u}}{u}\right)^{\frac{(n-1)(1-p)}{n}}\left(\frac{\Bar{r}}{r}\right)^{\frac{(n-1)(q-n-1)-n\alpha}{n}}\right)
    u \Bar{r}^{\alpha} P_{0,n} d\sigma\\
    \le & \frac{\alpha}{n}\int_{\mathbb{S}^n}
    \left(\Bar{u}^{p}\Bar{r}^{\alpha+n-1-q}\sqrt{\left(\frac{\Bar{r}}{\Bar{u}}\right)^2-1}-u^{p}r^{\alpha+n-1-q}\sqrt{\left(\frac{r}{u}\right)^2-1}\right)\\
    &\quad\quad \left(\sqrt{\left(\frac{\Bar{r}}{\Bar{u}}\right)^2-1}-\sqrt{\left(\frac{r}{u}\right)^2-1}\right) u\Bar{u}^{2-p}\Bar{r}^{q-n-1} P_{0,n} d\sigma.
\end{split}
\end{equation}

When $1\le p<n+1,\ q\le n+1$, we choose $\alpha=-\frac{(n-1)(n+1-q)}{n+1-p}\le0$. Let $\rho=\frac{\Bar{u}}{u}$ and $\Bar{\rho}=\frac{\Bar{r}}{r}$, we have for any $\rho,\Bar{\rho}>0$
\begin{align*}
    \left( \rho^{\frac{n+1-p}{n}}\Bar{\rho}^{\frac{q-n-1}{n}}-1\right)
    \left(1-\rho^{\frac{(n-1)(1-p)}{n}}\Bar{\rho}^{\frac{(n-1)(q-n-1)-n\alpha}{n}}\right)\ge 0.
\end{align*}


When $p=2$, then $\alpha=q-n-1$. Let $\xi=\frac{r}{u}$ and $\Bar{\xi}=\frac{\Bar{r}}{\Bar{u}}$, we have for any $1\le \xi,\Bar{\xi}\le \sqrt{2}$
\begin{align*}
    \left(\Bar{\xi}^{-2}\sqrt{\Bar{\xi}^2-1}-\xi^{-2}\sqrt{\xi^2-1}\right)\left(\sqrt{\Bar{\xi}^2-1}-\sqrt{\xi^2-1}\right)\ge 0,
\end{align*}
where we used the monotonic increasing property of $g(t)=t^{-2}\sqrt{t^2-1}$ for $1\le t\le\sqrt{2}$.

Consequently, the left hand side of (\ref{3.4}) is non-negative and the right hand side of (\ref{3.4}) is non-positive. Then both sides must be zero. It follows that $W=\lambda\Bar{W}$ and $\tau_i=\lambda^* \Bar{\tau}_i$, $1\le i\le n$, for some $\lambda,\lambda^*>0$. If $p\neq q$, we have $M=\Bar{M}$. If $p=q$, we have $M=\lambda \Bar{M}$.
\end{proof}

\begin{proof}[Proof of Theorem \ref{mainthm-2}]
Without loss of generality, we can assume $\psi(1,1)=1$, then $\psi(s,t)$ and $\eta(t)$ satisfy 
\begin{align}\label{3.5}
    \left(s-\psi(s,t)\right)\left(\eta(1)-\eta(t)\psi^{k-l-1}(s,t)\right)\le 0,
\end{align}
for any $t> s>0$. Suppose there exists a smooth strictly convex solution $M$ to (\ref{icp-eq}). Let $\Bar{M}$ be the unit sphere, then $\Bar{M}$ is a strictly convex solution to (\ref{icp-eq}).

Using the integral formulas (\ref{2.25}) and (\ref{2.26}), we have
\begin{equation*}
    \begin{aligned}
    & \int_{\mathbb{S}^n} \left[\eta(1)(P_{l+1}(W)-u P_{l}(W))+\eta(r)(u P_{k}(W)-P_{k-1}(W))\right]d\sigma\\
    =& \frac{(k-1)!}{n!}\int_{\mathbb{S}^n}\frac{\eta'(r)}{r}\sum \delta_{i_1,i_2\ldots,i_{n-k+1}}^{i_1,j_2,\ldots,j_{n-k+1}}\tau_{i_1}^2 h_{i_2 j_2}\cdots h_{i_{n-k+1} j_{n-k+1}} P_n(W)d\sigma.
    \end{aligned}
\end{equation*}
Note that for a fixed $z_0\in \mathbb{S}^n$, after choosing a local orthogonal frame such that $h_{ij}=\kappa_i\delta_{ij}$ at $z_0$, then we have at $z_0$
\begin{align*}
    \sum \delta_{i_1,i_2\ldots,i_{n-k+1}}^{i_1,j_2,\ldots,j_{n-k+1}}\tau_{i_1}^2 h_{i_2 j_2}\cdots h_{i_{n-k+1} j_{n-k+1}} 
    &=\sum \tau_{i_1}^2\kappa_{i_2}\cdots\kappa_{i_{n-k+1}}\ge 0.
\end{align*}
Then it follows from $\eta'(t)\le 0$ and Lemma \ref{Lem 2.1-2} that
\begin{align*}
    0 
    &\ge \int_{\mathbb{S}^n} \left[\eta(1)P_l(W)\left(\frac{P_{l+1}(W)}{P_l(W)} -u\right)+\eta(r)P_k(W) \left(u\frac{P_{k-1}(W)}{P_k(W)}-1\right)\right]d\sigma\\
    &\ge \int_{\mathbb{S}^n} \left[\eta(1)P_l(W)\left( \left(\frac{P_{k}(W)}{P_l(W)}\right)^{\frac{1}{k-l}} -u\right)+\eta(r)P_k(W) \left(u \left(\frac{P_{k}(W)}{P_l(W)}\right)^{-\frac{1}{k-l}}-1\right)\right]d\sigma\\
    &=\int_{\mathbb{S}^n} \left(u- \left(\frac{P_{k}(W)}{P_l(W)}\right)^{\frac{1}{k-l}}\right) \left(\eta(r) \left(\frac{P_{k}(W)}{P_l(W)}\right)^{\frac{k-l-1}{k-l}}-\eta(1)\right) P_l(W) d\sigma.
\end{align*}
Using (\ref{icp-eq}) and the assumption (\ref{3.5}), we obtain
\begin{equation}\label{3.6}
\begin{split}
    0 &\ge\int_{\mathbb{S}^n} \left(u -\psi(u,r)\right) \left(\eta(r) \psi^{k-l-1}(u,r)-\eta(1)\right) P_l(W) d\sigma\ge 0.
\end{split}
\end{equation}
Hence the inequalities in (\ref{3.6}) are both equalities. It follows that $W=\lambda I$, and then $M$ is a sphere. Moreover, if in addition $\eta'(t)<0$, then we have $\tau_i\equiv 0,\ 1\le i\le n$. Thus $r=u$ and $M$ is an origin-centred sphere.
\end{proof}

\section{Proof of Theorem \ref{mainthm-3}}\label{sec:4}

\begin{proof}[Proof of Theorem \ref{mainthm-3}]
Suppose there exists a smooth strictly convex solution $M$ to (\ref{icp-eq}). Let $\{e_1,\ldots,e_n,e_{n+1}\}$ be a local orthonormal frame at $x\in M$ such that $e_{n+1}$ is the unit inner normal vector of $M$ and $h_{ij}=\kappa_i\delta_{ij}$ at $x$, then we have $u_i=\kappa_i\langle x,e_i\rangle$ and $r_i=r^{-1}\langle x,e_i\rangle$ at $x$.

Denote $\Tilde{k}=n-k,\ \Tilde{l}=n-l$, $\Tilde{\psi}=\psi^{-1}$ and $P_i=P_i(\kappa)$ for $1\le i\le n$, we have  $0\le \Tilde{k}<\Tilde{l}\le n$, $\partial_1\Tilde{\psi}\le 0$ and $\partial_2\Tilde{\psi}\le 0$. Note that
$P_i(W)=P_i(\kappa^{-1})=\frac{P_{n-i}(\kappa)}{P_n(\kappa)}$, 
then (\ref{cond-3}) is equivalent to
\begin{equation}\label{4.1}
    \left(\frac{P_{\Tilde{l}}}{P_{\Tilde{k}}}\right)^{\frac{1}{\Tilde{l}-\Tilde{k}}}=\Tilde{\psi}(u,r).
\end{equation}

When $\Tilde{k}<\Tilde{l}-1$, by using Minkowski formulas (\ref{2.16}), Lemma \ref{Lem 2.1-2}, Proposition \ref{prop-1} and (\ref{4.1}), we obtain
\begin{equation}\label{4.2}
\begin{split}
    0 &\le \int_M (P_{\Tilde{l}-1}-P_{\Tilde{k}}^{\frac{1}{\Tilde{l}-\Tilde{k}}}P_{\Tilde{l}}^{\frac{\Tilde{l}-\Tilde{k}-1}{\Tilde{l}-\Tilde{k}}})d\mu
    +\int_M u(P_{\Tilde{k}+1}-P_{\Tilde{k}}^{\frac{\Tilde{l}-\Tilde{k}-1}{\Tilde{l}-\Tilde{k}}}P_{\Tilde{l}}^{\frac{1}{\Tilde{l}-\Tilde{k}}})\left(\frac{P_{\Tilde{l}}}{P_{\Tilde{k}}}\right)^{\frac{\Tilde{l}-\Tilde{k}-1}{\Tilde{l}-\Tilde{k}}}d\mu\\
    &=\int_M (P_{\Tilde{l}-1}-u P_{\Tilde{l}})d\mu
    +\int_M (u P_{\Tilde{k}+1}-P_{\Tilde{k}})\left(\frac{P_{\Tilde{l}}}{P_{\Tilde{k}}}\right)^{\frac{\Tilde{l}-\Tilde{k}-1}{\Tilde{l}-\Tilde{k}}} d\mu \\
    &\le \int_M (u P_{\Tilde{k}+1}-P_{\Tilde{k}})\left(\frac{P_{\Tilde{l}}}{P_{\Tilde{k}}}\right)^{\frac{\Tilde{l}-\Tilde{k}-1}{\Tilde{l}-\Tilde{k}}} d\mu
    =\frac{1}{(n-\Tilde{k})\binom{n}{\Tilde{k}}}\int_M T_{ij}^{\Tilde{k}}\langle x,e_i\rangle \left(\left(\frac{P_{\Tilde{l}}}{P_{\Tilde{k}}}\right)^{\frac{\Tilde{l}-\Tilde{k}-1}{\Tilde{l}-\Tilde{k}}}\right)_j d\mu\\
    &= \frac{\Tilde{l}-\Tilde{k}-1}{(n-\Tilde{k})\binom{n}{\Tilde{k}}}\int_M T_{ij}^{\Tilde{k}}\langle x,e_i\rangle (\Tilde{\psi}(u,r))_j\left(\frac{P_{\Tilde{l}}}{P_{\Tilde{k}}}\right)^{\frac{\Tilde{l}-\Tilde{k}-2}{\Tilde{l}-\Tilde{k}}} d\mu\\
    &= \frac{\Tilde{l}-\Tilde{k}-1}{(n-\Tilde{k})\binom{n}{\Tilde{k}}}\int_M T_{ij}^{\Tilde{k}}\langle x,e_i\rangle (\kappa_j\langle x,e_j\rangle \partial_1\Tilde{\psi} + r^{-1}\langle x,e_j\rangle \partial_2\Tilde{\psi}) \left(\frac{P_{\Tilde{l}}}{P_{\Tilde{k}}}\right)^{\frac{\Tilde{l}-\Tilde{k}-2}{\Tilde{l}-\Tilde{k}}} d\mu\\
    &\le 0.
\end{split}
\end{equation}

When $\Tilde{k}=\Tilde{l}-1\ge 1$, we obtain
\begin{equation}\label{4.3}
\begin{split}
    0 &\le \int_M (P_{\Tilde{k}}^2-P_{\Tilde{k}-1} P_{\Tilde{k}+1})\frac{1}{P_{\Tilde{k}}} d\mu 
    =\int_M (P_{\Tilde{k}}-u P_{\Tilde{k+1}})d\mu
    +\int_M (u P_{\Tilde{k}}-P_{\Tilde{k}-1})\frac{P_{\Tilde{k}+1}}{P_{\Tilde{k}}} d\mu \\
    &= \int_M (u P_{\Tilde{k}}-P_{\Tilde{k}-1})\frac{P_{\Tilde{k}+1}}{P_{\Tilde{k}}} d\mu
    =\frac{1}{(n-\Tilde{k}+1)\binom{n}{\Tilde{k}-1}}\int_M T_{ij}^{\Tilde{k}-1}\langle x,e_i\rangle \left(\frac{P_{\Tilde{k}+1}}{P_{\Tilde{k}}}\right)_j d\mu\\
    &= \frac{1}{(n-\Tilde{k}+1)\binom{n}{\Tilde{k}-1}}\int_M T_{ij}^{\Tilde{k}-1}\langle x,e_i\rangle (\kappa_j\langle x,e_j\rangle \partial_1\Tilde{\psi} + r^{-1}\langle x,e_j\rangle \partial_2\Tilde{\psi}) d\mu\\
    &\le 0. 
\end{split}
\end{equation}

When $\Tilde{k}=\Tilde{l}-1=0$, since $n\ge 2$, we obtain
\begin{equation}\label{4.4}
\begin{split}
    0 &\le \int_M u(P_{1}^2- P_{2}) d\mu 
    =\int_M (P_{1}-u P_{2})d\mu
    +\int_M (u P_{1}-1)P_{1} d\mu \\
    &= \int_M (u P_{1}-1)P_{1} d\mu
    =\frac{1}{n}\int_M \delta_{ij}\langle x,e_i\rangle \left(P_1\right)_j d\mu\\
    &= \frac{1}{n}\int_M \sum_i \langle x,e_i\rangle (\kappa_i\langle x,e_i\rangle \partial_1\Tilde{\psi} + r^{-1}\langle x,e_i\rangle \partial_2\Tilde{\psi}) d\mu\\
    &\le 0. 
\end{split}
\end{equation}

Therefore, the above inequalities in (\ref{4.2}), (\ref{4.3}) and (\ref{4.4}) are all equalities. It follows that $W=\lambda I$, $\langle x,e_i\rangle\equiv 0\ (1\le i\le n)$, and then $M$ is an origin-centred sphere.
\end{proof}

$\ $

Furthermore, we study the following equation
\begin{equation}\label{comb-eq}
    \sum\limits_{s=1}^{n-k}f_s(u,r)\frac{P_{k+s}(\kappa)}{P_{k}(\kappa)}+\sum\limits_{s=1}^{n-k}g_s(u,r)\frac{P_{k-1+s}(\kappa)}{P_{k-1}(\kappa)}=\psi(u.r),
\end{equation}
where $\kappa$ are the principal curvatures of the hypersurface. We can prove the following uniqueness result, which generalizes the result in \cite{Si67}.

\begin{prop}\label{mainthm-4}
Let $n\ge 2$ and $0< k< n$. Suppose $\psi:(0,\infty)\times(0,\infty)\to (0,\infty)$ is a $C^1$ function with $\partial_1\psi\le 0,\ \partial_2\psi\le 0$ and at least one of these inequalities is strict. For $1\le s\le n-k$, suppose $f_s,\ g_s:(0,\infty)\times(0,\infty)\to (0,\infty)$ are $C^1$ functions with $f_s+g_s=\phi_s$ and $\partial_1\phi_s\ge 0,\ \partial_2\phi_s\ge 0$. Then the smooth strictly convex solution $M$ to (\ref{comb-eq}) must be an origin-centred sphere.    
\end{prop}

\begin{proof}[Proof of Proposition \ref{mainthm-4}]
Suppose there exists a smooth strictly convex solution $M$ to (\ref{comb-eq}). Let $\{e_1,\ldots,e_n,e_{n+1}\}$ be a local orthonormal frame at $x\in M$ such that $e_{n+1}$ is the unit inner normal vector of $M$ and $h_{ij}=\kappa_i\delta_{ij}$ at $x$. Denote $P_i=P_i(\kappa)$.

Using Minkowski formulas (\ref{2.16}), Lemma \ref{Lem 2.1-2}, Proposition \ref{prop-1} and (\ref{comb-eq}), we obtain
\begin{equation}\label{4.5}
\begin{split}
    0 &\le \sum\limits_{s=1}^{n-k}\int_M (P_k P_{k-1+s}-P_{k-1}P_{k+s})\left(\frac{f_s}{P_k}+u\frac{g_s}{P_{k-1}}\right) d\mu \\
    &= \sum\limits_{s=1}^{n-k}\int_M (P_{k-1+s}-u P_{k+s})(f_s+g_s) d\mu
    - \sum\limits_{s=1}^{n-k}\int_M (P_{k-1}-u P_{k})\left(\frac{f_s P_{k+s}}{P_k}+u\frac{g_s P_{k-1+s}}{P_{k-1}}\right) d\mu\\
    &= \sum\limits_{s=1}^{n-k}\int_M (P_{k-1+s}-u P_{k+s})\phi_s d\mu
    - \int_M (P_{k-1}-u P_{k})\psi d\mu\\
    &= -\sum\limits_{s=1}^{n-k}\frac{\int_M T_{ij}^{k-1+s}\langle x,e_i\rangle (\phi_s)_j d\mu}{(n-k+1-s)\binom{n}{k-1+s}}
    +\frac{\int_M T_{ij}^{k-1}\langle x,e_i\rangle (\psi)_j d\mu}{(n-k+1)\binom{n}{k-1}} \\
    &\le 0.
\end{split}
\end{equation}
Hence the above inequalities in (\ref{4.5}) are both equalities. It follows that $W=\lambda I$, $\langle x,e_i\rangle\equiv 0\ (1\le i\le n)$, and then $M$ is an origin-centred sphere.
\end{proof}

\section{Proofs of Theorem \ref{mainthm-6} and Theorem \ref{mainthm-7}}\label{sec:5}

\begin{proof}[Proof of Theorem \ref{mainthm-6}]
Suppose there exists a smooth strictly origin-symmetric convex solution $M$ to (\ref{p,q-M-eq}). Then we have $dV_n=u^p |X|^{n+1-q}d\sigma$.  

Since $M$ is origin-symmetric, then $\int_{\mathbb{S}^n}|X|^{\alpha}X dV_n=0$ for any $\alpha\in\mathbb{R}$. Using (\ref{2.29}), we have
\begin{equation}\label{5.1}
    n\int_{\mathbb{S}^n}|X|^{2\alpha+2}dV_n
    \le \int_{\mathbb{S}^n}|X|^{2\alpha}u(\Delta u+n u) dV_n
    +(\alpha^2+2\alpha)\int_{\mathbb{S}^n}|X|^{2\alpha-1} u \langle\nabla u,\nabla |X|\rangle dV_n.
\end{equation}
Note that $|X|^2=u^2+|\nabla u|^2$ and
\begin{align}\label{5.2}
    \langle \nabla u,\nabla|X|^2\rangle
    =\sum_j 2u_j(uu_j+\sum_i u_i u_{ij})
    =\sum_{i,j}2u_i u_j(u_{ij}+u\delta_{ij})
    \ge c|\nabla u|^2,
\end{align}
where $c>0$ depends on $M$. By integration by parts, we have
\begin{equation}\label{5.3}
\begin{split}
     & \int_{\mathbb{S}^n}|X|^{2\alpha}u\Delta u dV_n
    =\int_{\mathbb{S}^n}|X|^{2\alpha+n+1-q}u^{p+1}\Delta u d\sigma\\
    =& -(p+1)\int_{\mathbb{S}^n}|X|^{2\alpha+n+1-q}u^{p}|\nabla u|^2 d\sigma
    -(2\alpha+n+1-q)\int_{\mathbb{S}^n}|X|^{2\alpha+n-q}u^{p+1}\langle \nabla u,\nabla|X|\rangle d\sigma\\
    =& -(p+1)\int_{\mathbb{S}^n}|X|^{2\alpha}|\nabla u|^2 dV_n
    -(2\alpha+n+1-q)\int_{\mathbb{S}^n}|X|^{2\alpha-1}u\langle \nabla u,\nabla|X|\rangle d\sigma.
\end{split}
\end{equation}
Substituting (\ref{5.3}) into (\ref{5.1}), we have
\begin{equation}\label{5.1-*}
    (n+p+1)\int_{\mathbb{S}^n}|X|^{2\alpha}|\nabla u|^2 dV_n
    \le (\alpha^2+q-n-1)\int_{\mathbb{S}^n}|X|^{2\alpha-1} u \langle\nabla u,\nabla |X|\rangle dV_n.
\end{equation}

When $p\ge -n-1$ and $q\le n+1$ with at least one of these being strict, choosing $\alpha=0$, it follows from (\ref{5.2}) and (\ref{5.1-*}) that $\nabla u\equiv 0$ on $\mathbb{S}^n$. Hence $M$ is an origin-centred sphere.

When $p> -n-1$ and $q> n+1$, using the same notation as in the proof of Lemma \ref{Lem 2.4-2}, we have
\begin{equation}\label{5.4-*}
\begin{split}
    & \int_{\mathbb{S}^n}|X|^{2\alpha-1} u \langle\nabla u,\nabla |X|\rangle dV_n
    = \int_{\mathbb{S}^n}|X|^{2\alpha-2} u \sum_i \lambda_i u_i^2 dV_n \\
    \le& \int_{\mathbb{S}^n}|X|^{2\alpha-2} u |\nabla u|^2 (\Delta u+ nu) dV_n.
\end{split}
\end{equation}
Assume that $\alpha+\frac{n+1-q}{2}\ge 0$, by (\ref{5.2}), we calculate
\begin{equation}\label{5.4}
\begin{split}
    & \int_{\mathbb{S}^n}|X|^{2\alpha-2}u |\nabla u|^2 \Delta u dV_n
    =\int_{\mathbb{S}^n}|X|^{2\alpha+n-1-q}u^{p+1} |\nabla u|^2 \Delta u d\sigma\\
    =& -(p+1)\int_{\mathbb{S}^n}|X|^{2\alpha+n-1-q}u^{p}|\nabla u|^4 d\sigma
    -\int_{\mathbb{S}^n}|X|^{2\alpha+n-1-q} u^{p+1}\langle \nabla u,\nabla|\nabla u|^2\rangle  d\sigma\\
    &\quad -\left(\alpha+\frac{n-1-q}{2}\right)\int_{\mathbb{S}^n}|X|^{2\alpha+n-3-q}u^{p+1}|\nabla u|^2\langle \nabla u,\nabla|X|^2\rangle d\sigma\\
    =& -(p+1)\int_{\mathbb{S}^n}|X|^{2\alpha+n-1-q}u^{p}|\nabla u|^4 d\sigma
    +2\int_{\mathbb{S}^n}|X|^{2\alpha+n-1-q}u^{p+2}|\nabla u|^2 d\sigma\\
    &\quad -\int_{\mathbb{S}^n}\left(|X|^2+\left(\alpha+\frac{n-1-q}{2}\right)|\nabla u|^2\right)|X|^{2\alpha+n-3-q}u^{p+1}\langle \nabla u,\nabla|X|^2\rangle d\sigma\\
    \le& -(p+1)\int_{\mathbb{S}^n}|X|^{2\alpha+n-1-q}u^{p}|\nabla u|^4 d\sigma
    +2\int_{\mathbb{S}^n}|X|^{2\alpha+n-1-q}u^{p+2}|\nabla u|^2 d\sigma\\
    =& -(p+1)\int_{\mathbb{S}^n}|X|^{2\alpha-2}|\nabla u|^4 dV_n
    +2\int_{\mathbb{S}^n}|X|^{2\alpha-2}u^{2}|\nabla u|^2 dV_n.
\end{split}
\end{equation}
It follows from (\ref{5.1-*}), (\ref{5.4-*}) and (\ref{5.4}) that
\begin{equation}\label{5.5}
\begin{split}
     0 & \le  \left(-(n+p+1)-(p+1)(\alpha^2+q-n-1)\right)\int_{\mathbb{S}^n}|X|^{2\alpha-2}|\nabla u|^4 dV_n\\
     &\quad +\left(-(n+p+1)+(n+2)(\alpha^2+q-n-1)\right)\int_{\mathbb{S}^n}|X|^{2\alpha-2} u^2 |\nabla u|^2 dV_n.
\end{split}
\end{equation}
Notice that $-(p+1)< n+2$. We choose $\alpha=\sqrt{n+1-q+\frac{n+p+1}{n+2}}$ according to the provided conditions for $p$ and $q$. This choice ensures that $\alpha+\frac{n+1-q}{2}\ge 0$. Consequently, (\ref{5.5}) becomes
\begin{equation}\label{5.6}
	 0  \le  -\frac{(n+p+3)(n+p+1)}{n+2}\int_{\mathbb{S}^n}|X|^{2\alpha-2}|\nabla u|^4 dV_n\le 0.
\end{equation}
Hence $\nabla u\equiv 0$ on $\mathbb{S}^n$. The proof of Case (\ref{thm-6i}) is complete. 

Moreover, using the duality relation of $L_p$ dual Minkowski problems in \cite[Theorem 7.1]{CHZ19}, we have that the solutions to (\ref{p,q-M-eq}) with parameters $(p,q)$ are in one-to-one correspondence with the solutions to (\ref{p,q-M-eq}) with parameters $(-q,-p)$. Combining this with Case (\ref{thm-6i}), we deduce the proof of Case (\ref{thm-6ii}). We complete the proof of Theorem \ref{mainthm-6}.
\end{proof}

\begin{proof}[Proof of Theorem \ref{mainthm-7}]
Suppose there exists a smooth strictly convex solution $M$ to (\ref{p,q-M-eq}). Let $\alpha$ be a constant to be determined such that $\alpha+\frac{n+1-q}{2}\ge 0$. By (\ref{2.29}), we have
\begin{equation}\label{5.7} 
\begin{split}
    n\int_{\mathbb{S}^n}|X|^{2\alpha+2}dV_n
    & \le 
    n\frac{|\int_{\mathbb{S}^n}|X|^{\alpha}X dV_n|^2}{\int_{\mathbb{S}^n}dV_n}
    +\int_{\mathbb{S}^n}|X|^{2\alpha}u(\Delta u+n u) dV_n\\
    &\quad +(\alpha^2+2\alpha)\int_{\mathbb{S}^n}|X|^{2\alpha-1} u\langle \nabla u,\nabla|X|\rangle dV_n.
\end{split}
\end{equation}

Using the same procedure as Theorem \ref{mainthm-6}, we can obtain (\ref{5.3}), (\ref{5.4-*}) and (\ref{5.4}) again. Since $q-n-1\ge 0$, we choose $\alpha=q-n-1$. It follows from the divergence theorem that
\begin{equation*}
\begin{split}
    \int_{\mathbb{S}^n}|X|^{\alpha}X dV_n
    &=\int_{\mathbb{S}^n} u^p X d\sigma
    =\int_{\mathbb{S}^n} u^p \nabla u d\sigma+\int_{\mathbb{S}^n} u^{p+1} z d\sigma(z)\\
    &=\frac{n+p+1}{n}\int_{\mathbb{S}^n} u^p \nabla u d\sigma
    =\frac{n+p+1}{n}\int_{\mathbb{S}^n} |X|^{\alpha} \nabla u dV_n.
\end{split}
\end{equation*}
Combining this with Cauchy-Schwarz inequality, we have
\begin{equation}\label{5.8}
\begin{split}
    n\frac{|\int_{\mathbb{S}^n}|X|^{\alpha}X dV_n|^2}{\int_{\mathbb{S}^n}dV_n}
     \le & \frac{(n+p+1)^2}{n} \frac{(\int_{\mathbb{S}^n}|X|^{\alpha}|\nabla u| dV_n)^2}{\int_{\mathbb{S}^n}dV_n} \\
     \le & \frac{(n+p+1)^2}{n} \int_{\mathbb{S}^n}|X|^{2\alpha}|\nabla u|^2 dV_n.
\end{split}
\end{equation}
Substituting (\ref{5.3}), (\ref{5.4-*}) (\ref{5.4}) and (\ref{5.8}) into (\ref{5.7}), we obtain
\begin{equation}\label{5.9}
\begin{split}
     0 & \le  \left(\frac{(p+1)(n+p+1)}{n}-(p+1)(q-n-1)(q-n)\right)\int_{\mathbb{S}^n}|X|^{2\alpha-2}|\nabla u|^4 dV_n\\
     &\quad +\left(\frac{(p+1)(n+p+1)}{n}+(n+2)(q-n-1)(q-n)\right)\int_{\mathbb{S}^n}|X|^{2\alpha-2} u^2 |\nabla u|^2 dV_n.
\end{split}
\end{equation}
Due to $(p+1)(n+p+1)\le 0$ and $n+1\le q\le n+\frac{1}{2}+\sqrt{\frac{1}{4}-\frac{(p+1)(n+p+1)}{n(n+2)}}$, we obtain
\begin{equation}\label{5.10}
\begin{split}
    & \frac{(p+1)(n+p+1)}{n}-(p+1)(q-n-1)(q-n)\\
    \le\ & \frac{(p+1)(n+p+1)}{n}+(n+2)(q-n-1)(q-n)\\
    \le\ & 0.
\end{split}
\end{equation}
When $-n-1<p<-1$, then at least one of inequalities in (\ref{5.10}) is strict. It follows from (\ref{5.9}) that $\nabla u\equiv 0$ on $\mathbb{S}^n$. When $p=-1$, then $q=n+1$ and $\alpha=0$. In this case, (\ref{5.9}) becomes an equality. It follows from the equality case of (\ref{5.8}) that $\nabla u\equiv 0$ on $\mathbb{S}^n$. The proof of Case (\ref{thm-7i}) is complete. Using the duality relation in \cite[Theorem 7.1]{CHZ19}, we deduce the proof of Case (\ref{thm-7ii}) and complete the proof of Theorem \ref{mainthm-7}.
\end{proof}


\appendix
\section{Uniqueness of the $L_p$ dual Christoffel-Minkowski problem}\label{appendix}

In this appendix, by modifying the test functions in \cite[Lemma 3.1]{IM23} for $k<n$, we provide a uniqueness result of origin-symmetric solutions to the isotropic $L_p$ dual Christoffel-Minkowski problem. Denote $dV_k=u\sigma_k d\sigma$. We first present a generalization of Lemma \ref{Lem 2.4-2}, which reduces to \cite[Lemma 3.2]{IM23} when $\alpha=0$.

\begin{lem}\label{lem a}
Let $1\le k\le n$. Assume that $\alpha\ge 0$. Then we have
\begin{equation}\label{a.1} 
\begin{split}
    k\int_{\mathbb{S}^n}|X|^{2\alpha+2}dV_k
    & \le 
    k\frac{|\int_{\mathbb{S}^n}|X|^{\alpha}X dV_k|^2}{\int_{\mathbb{S}^n}dV_k}
    +\int_{\mathbb{S}^n}|X|^{2\alpha}u \left(\sigma_1-(k+1)\frac{\sigma_{k+1}}{\sigma_k}\right) dV_k\\
    &\quad +(\alpha^2+2\alpha)\int_{\mathbb{S}^n}|X|^{2\alpha-2}|\nabla u|^2 u \left(\sigma_1-(k+1)\frac{\sigma_{k+1}}{\sigma_k}\right) dV_n. 
\end{split}
\end{equation}
\end{lem}

\begin{proof}
Let $\{E_l\}_{l=1}^{n+1}$ be an orthonormal basis of $\mathbb{R}^{n+1}$. Suppose $\{e_i\}_{i=1}^n$ is a local orthonormal frame for $\mathbb{S}^n$ such that $(u_{ij}+u\delta_{ij})(z_0)=\lambda_i(z_0)\delta_{ij}$. For $l=1,\ldots,n+1$, we define the functions $f_l:\mathbb{S}^n\to\mathbb{R}$
\begin{equation}
    f_l(z)=|X(z)|^{\alpha}\langle X(z),E_l\rangle-\frac{\int_{\mathbb{S}^n}|X(z)|^{\alpha}\langle X(z),E_l\rangle dV_k}{\int_{\mathbb{S}^n}dV_k}.
\end{equation}
Since $\int_{\mathbb{S}^n}f_l dV_k=0$ for $1\le l\le n+1$, applying Lemma \ref{Lem 2.4-1} to $f_l$ and summing over $l$, we have
\begin{align*}
     k\sum_l \int_{\mathbb{S}^n}f_l^2 dV_k
    =k\left[\int_{\mathbb{S}^n}|X|^{2\alpha+2}dV_k-\frac{|\int_{\mathbb{S}^n}|X|^{\alpha}X dV_k|^2}{\int_{\mathbb{S}^n}dV_k}\right]
    \le \sum_l \int_{\mathbb{S}^n}u^2 \sigma_k^{ij}\nabla_i f_l\nabla_j f_l d\sigma.
\end{align*}
Note that $\sum_i \langle e_i,X\rangle^2=|\nabla u|^2$ and $\sum_i \frac{\partial \sigma_k}{\partial \lambda_i}\lambda_i^2=\sigma_1\sigma_k-(k+1)\sigma_{k+1}$. It follows from $(u_{ij}+u\delta_{ij})(z_0)=\lambda_i(z_0)\delta_{ij}$ that $\nabla_i X=\sum_j (u_{ij}+u\delta_{ij})e_j=\lambda_i e_i$ at $z_0$. Since $\alpha\ge 0$, we have
\begin{align*}
    \sum_{l,i,j}\sigma_k^{ij}\nabla_i f_l\nabla_j f_l
    &=\sum_{l,i}\frac{\partial \sigma_k}{\partial \lambda_i}(\nabla_i(|X|^{\alpha})\langle X,E_l\rangle+|X|^{\alpha}\langle \nabla_i X,E_l\rangle)^2\\
    &=\sum_{l,i}\frac{\partial \sigma_k}{\partial \lambda_i}(\alpha|X|^{\alpha-2}\langle \lambda_i e_i,X\rangle\langle X,E_l\rangle+|X|^{\alpha}\langle \lambda_i e_i,E_l\rangle)^2\\
    &=\sum_i \frac{\partial \sigma_k}{\partial \lambda_i}\lambda_i^2 (|X|^{2\alpha}+(\alpha^2+2\alpha)|X|^{2\alpha-2}\langle e_i,X\rangle^2)\\
    &\le (|X|^{2\alpha}+(\alpha^2+2\alpha)|X|^{2\alpha-2}|\nabla u|^2)(\sigma_1\sigma_k-(k+1)\sigma_{k+1}).
\end{align*}
Therefore, we obtain
\begin{align*}
    & k\left[\int_{\mathbb{S}^n}|X|^{2\alpha+2}dV_k-\frac{|\int_{\mathbb{S}^k}|X|^{\alpha}X dV_k|^2}{\int_{\mathbb{S}^n}dV_n}\right]\\
    \le & \int_{\mathbb{S}^n}u^2 (|X|^{2\alpha}+(\alpha^2+2\alpha)|X|^{2\alpha-2}|\nabla u|^2)(\sigma_1\sigma_k-(k+1)\sigma_{k+1})d\sigma\\
    =& \int_{\mathbb{S}^n}u (|X|^{2\alpha}+(\alpha^2+2\alpha)|X|^{2\alpha-2}|\nabla u|^2) \left(\sigma_1-(k+1)\frac{\sigma_{k+1}}{\sigma_k}\right) dV_k.
\end{align*}
This completes the proof of Lemma \ref{lem a}.
\end{proof}


Now we consider the following isotropic $L_p$ dual Christoffel-Minkowski problem
\begin{equation}\label{a.3}
    u^{1-p}r^{q-k-1}\sigma_k(W)=1\quad \text{on}\ \mathbb{S}^n.    
\end{equation}
Corollary \ref{cor-6} states the uniqueness of solutions to (\ref{a.3}) when $p\ge 1$ and $q\le k+1$ with at least one of these inequalities being strict. As for the case with the origin-symmetric assumption, in \cite[Theorem 1.6]{IM23}, by choosing $\varphi(s,t)=s^p t^{k+1-q}$, we obtain
\begin{thm}[\cite{IM23}]\label{thm-IM}
Let $n\ge 2$ and $1\le k \le n-1$. Suppose $p\ge 1-k$ and $q\le k+1$. Then the smooth strictly convex origin-symmetric solution $M$ to (\ref{a.3}) must be an origin-centred sphere.
\end{thm}

As an extension of the above theorem, by applying Lemma \ref{lem a} for $k<n$, we have the following result:
\begin{thm}\label{mainthm-a}
Let $n\ge 2$ and $1\le k \le n-1$. Suppose $p\ge 1-k$ and $q\le k+1+ 2\alpha_*$, where $\alpha_*$ is the maximum value satisfying the following conditions:
\begin{equation}\label{a.4}
\begin{split}
    \left\{\begin{array}{l}
    0\le \alpha\le 1,\\
    -2\alpha(p+1)+(\alpha^2+2\alpha)(3p+1)\le p+k-1,\\
    2\alpha(n+2)+(\alpha^2+2\alpha)(2n+k+6)\le p+k-1.\end{array}\right.
\end{split}
\end{equation}
Then the smooth strictly convex origin-symmetric solution $M$ to (\ref{a.3}) must be an origin-centred sphere.
\end{thm}

\begin{rem}
    When $\alpha_*=0$, Theorem \ref{mainthm-a} reduces to Theorem \ref{thm-IM}.
\end{rem}

\begin{proof}[Proof of Theorem \ref{mainthm-a}]
Suppose there exists a smooth strictly origin-symmetric convex solution $M$ to (\ref{a.3}). For convenience, we denote $\tau=q-k-1$. Then we have $dV_k=u^p |X|^{-\tau}d\sigma$.

Since $\int_{\mathbb{S}^n}|X|^{\alpha}X dV_k=0$ for any $\alpha\in\mathbb{R}$, it follows from (\ref{a.1}) that for $\alpha\ge 0$
\begin{equation}\label{a.5} 
\begin{split}
    k\int_{\mathbb{S}^n}|X|^{2\alpha+2}dV_k
    & \le 
    \int_{\mathbb{S}^n}|X|^{2\alpha}u \left(\Delta u+n u-(k+1)\frac{\sigma_{k+1}}{\sigma_k}\right) dV_k\\
    &\quad +(\alpha^2+2\alpha)\int_{\mathbb{S}^n}|X|^{2\alpha-2}|\nabla u|^2 u \left(\Delta u+n u-(k+1)\frac{\sigma_{k+1}}{\sigma_k}\right) dV_k. 
\end{split}
\end{equation}
Since $\alpha_*\ge 0$ and $\alpha_*-\frac{\tau}{2}\ge 0$, by using (\ref{5.2}), we obtain
\begin{equation}\label{a.6}
\begin{split}
     & \int_{\mathbb{S}^n}|X|^{2\alpha_*}u\Delta u dV_k
    = \int_{\mathbb{S}^n}|X|^{2\alpha_*-\tau}u^{p+1}\Delta u d\sigma\\
    =& -(p+1)\int_{\mathbb{S}^n}|X|^{2\alpha_*-\tau}u^{p}|\nabla u|^2 d\sigma 
    -\left(\alpha_*-\frac{\tau}{2}\right)\int_{\mathbb{S}^n}|X|^{2\alpha_*-\tau}u^{p+1}\langle \nabla u,\nabla|X|^2\rangle d\sigma\\
    \le& -(p+1)\int_{\mathbb{S}^n}|X|^{2\alpha_*-\tau}u^{p}|\nabla u|^2 d\sigma\\
    =& -(p+1)\int_{\mathbb{S}^n}|X|^{2\alpha_*}|\nabla u|^2 dV_k
\end{split}
\end{equation}
and
\begin{equation}\label{a.7}
\begin{split}
    & \int_{\mathbb{S}^n}|X|^{2\alpha_*-2}|\nabla u|^2 u\Delta u dV_k
    =\int_{\mathbb{S}^n}|X|^{2\alpha_*-\tau-2}|\nabla u|^2 u^{p+1}\Delta u d\sigma\\
    =& -(p+1)\int_{\mathbb{S}^n}|X|^{2\alpha_*-\tau-2}u^{p}|\nabla u|^4 d\sigma
    -\int_{\mathbb{S}^n}|X|^{2\alpha_*-\tau-2}u^{p+1}\langle \nabla u,\nabla|\nabla u|^2\rangle  d\sigma\\
    &\quad -\left(\alpha_*-\frac{\tau}{2}-1\right)\int_{\mathbb{S}^n}|X|^{2\alpha_*-\tau-4}u^{p+1}|\nabla u|^2\langle \nabla u,\nabla|X|^2\rangle d\sigma\\
    =& -(p+1)\int_{\mathbb{S}^n}|X|^{2\alpha_*-\tau-2}u^{p}|\nabla u|^4 d\sigma
    +2\int_{\mathbb{S}^n}|X|^{2\alpha_*-\tau-2}u^{p+2}|\nabla u|^2 d\sigma\\
    &\quad -\int_{\mathbb{S}^n}\left(|X|^2+\left(\alpha_*-\frac{\tau}{2}-1\right)|\nabla u|^2\right)|X|^{2\alpha_*-\tau-4}u^{p+1}\langle \nabla u,\nabla|X|^2\rangle d\sigma\\
    \le & -(p+1)\int_{\mathbb{S}^n}|X|^{2\alpha_*-\tau-2}u^{p}|\nabla u|^4 d\sigma
    +2\int_{\mathbb{S}^n}|X|^{2\alpha_*-\tau-2} u^{p+2}|\nabla u|^2 d\sigma\\
    =& -(p+1)\int_{\mathbb{S}^n}|X|^{2\alpha-2}|\nabla u|^4 dV_k
    +2\int_{\mathbb{S}^n}|X|^{2\alpha_*-2}u^{2}|\nabla u|^2 dV_k.
\end{split}
\end{equation}

Suppose that $\{e_i\}_{i=1}^{n}$ is a local orthonormal frame for $\mathbb{S}^n$ such that $(u_{ij}+u\delta_{ij})(z_0)=\lambda_i(z_0)\delta_{ij}$, then $\langle X,X_i\rangle=\langle X,\lambda_i e_i\rangle=\lambda_i u_i$. It is well-known that $\sigma_{k+1}^{ii}=\sigma_k-\lambda_i\sigma_k^{ii}$ at $z_0$. Then we have at $z_0$
\begin{align*}
    0\le \sum_{i,j} u_i u_j\sigma_{k+1}^{ij}=\sum_i u_i^2(\sigma_k-\lambda_i\sigma_k^{ii})\le |\nabla u|^2\sigma_k,
\end{align*}
and
\begin{align*}
    \sum_{i,j} \langle X,X_i\rangle u_j\sigma_{k+1}^{ij}=\sum_i \lambda_i u_i^2(\sigma_k-\lambda_i\sigma_k^{ii})\le |\nabla u|^2\sigma_1\sigma_k.
\end{align*}
Note that $\nabla_i\sigma_{k+1}^{ij}=0$ and $(k+1)\sigma_{k+1}=\sigma_{k+1}^{ij}(u_{ij}+u\delta_{ij})$. Since $0\le \alpha_*\le 1$, by using (\ref{a.7}), we calculate
\begin{equation}\label{a.8}
\begin{split}
    & (k+1)\int_{\mathbb{S}^n}|X|^{2\alpha_*} u \frac{\sigma_{k+1}}{\sigma_k} dV_k\\
    =& (k+1)\int_{\mathbb{S}^n}|X|^{2\alpha_*} u^2 \sigma_{k+1} d\sigma 
    =\int_{\mathbb{S}^n}|X|^{2\alpha_*} u^2 \sigma_{k+1}^{ij}(u_{ij}+u\delta_{ij}) d\sigma \\
    =& -2\int_{\mathbb{S}^n}|X|^{2\alpha_*} u u_i u_j \sigma_{k+1}^{ij} d\sigma 
    -2\alpha_*\int_{\mathbb{S}^n}|X|^{2\alpha_*-2}u^2 \langle X,X_i\rangle  u_j \sigma_{k+1}^{ij} d\sigma 
    +\int_{\mathbb{S}^n}|X|^{2\alpha_*} u^3 \sigma_{k+1}^{ij}\delta_{ij} d\sigma \\
    \ge & -2\int_{\mathbb{S}^n}|X|^{2\alpha_*}|\nabla u|^2 dV_k
    -2\alpha_*\int_{\mathbb{S}^n}|X|^{2\alpha_*-2} u |\nabla u|^2 \sigma_1 dV_k 
    +(n-k)\int_{\mathbb{S}^n}|X|^{2\alpha_*} u^2 dV_k\\
    \ge & -2\int_{\mathbb{S}^n}|X|^{2\alpha_*}|\nabla u|^2 dV_k
    +2\alpha_* (p+1)\int_{\mathbb{S}^n}|X|^{2\alpha_*-2}|\nabla u|^4 dV_k
    \\
    &\quad -2\alpha_* (n+2)\int_{\mathbb{S}^n}|X|^{2\alpha_*-2}u^{2}|\nabla u|^2 dV_k
    +(n-k)\int_{\mathbb{S}^n}|X|^{2\alpha_*} u^2 dV_k
\end{split}
\end{equation}
and
\begin{equation}\label{a.9}
\begin{split}
    & (k+1)\int_{\mathbb{S}^n}|X|^{2\alpha_*-2}|\nabla u|^2 u \frac{\sigma_{k+1}}{\sigma_k} dV_k
    =\int_{\mathbb{S}^n}|X|^{2\alpha_*-2}|\nabla u|^2 u^2 \sigma_{k+1}^{ij}(u_{ij}+u\delta_{ij}) d\sigma \\
    =& -2\int_{\mathbb{S}^n}|X|^{2\alpha_*-2}|\nabla u|^2 u u_i u_j \sigma_{k+1}^{ij} d\sigma 
    -2(\alpha_*-1)\int_{\mathbb{S}^n}|X|^{2\alpha_*-4}|\nabla u|^2 u^2 \langle X,X_i\rangle  u_j \sigma_{k+1}^{ij} d\sigma\\
    &\quad -\int_{\mathbb{S}^n}|X|^{2\alpha_*-2} u^2(|\nabla u|^2)_i u_j \sigma_{k+1}^{ij} d\sigma
    +\int_{\mathbb{S}^n}|X|^{2\alpha_*-2}|\nabla u|^2 u^3 \sigma_{k+1}^{ij}\delta_{ij} d\sigma \\
    =& -2\int_{\mathbb{S}^n}|X|^{2\alpha_*-2}|\nabla u|^2 u u_i u_j \sigma_{k+1}^{ij} d\sigma 
    -2\int_{\mathbb{S}^n}\left((\alpha_*-1)|\nabla u|^2+|X|^2\right)|X|^{2\alpha_*-4} u^2 \langle X,X_i\rangle  u_j \sigma_{k+1}^{ij} d\sigma\\
    &\quad +2\int_{\mathbb{S}^n}|X|^{2\alpha_*-2} u^3 u_i u_j \sigma_{k+1}^{ij} d\sigma
    +(n-k)\int_{\mathbb{S}^n}|X|^{2\alpha_*-2}|\nabla u|^2 u^2 dV_k \\
    \ge & -2\int_{\mathbb{S}^n}|X|^{2\alpha_*-2}|\nabla u|^4 dV_k
    -2\int_{\mathbb{S}^n}\left((\alpha_*-1)|\nabla u|^2+|X|^2\right)|X|^{2\alpha_*-4} u |\nabla u|^2 \sigma_1  dV_k\\
    &\quad +(n-k)\int_{\mathbb{S}^n}|X|^{2\alpha_*-2}|\nabla u|^2 u^2 dV_k\\
    \ge &\ 2p\int_{\mathbb{S}^n}|X|^{2\alpha_*-2}|\nabla u|^4 dV_k
    -(n+k+4)\int_{\mathbb{S}^n}|X|^{2\alpha_*-2}|\nabla u|^2 u^2 dV_k.
\end{split}
\end{equation}
Substituting (\ref{a.6}), (\ref{a.7}), (\ref{a.8}) and (\ref{a.9}) into (\ref{a.5}), we obtain
\begin{equation}\label{a.10}
\begin{split}
    0 &\le \left(1-k-p-2\alpha_* (p+1)+(\alpha_*^2+2\alpha_*)(3p+1)\right)\int_{\mathbb{S}^n}|X|^{2\alpha_*-2}|\nabla u|^4 dV_k\\
    &\quad +\left(1-k-p+2\alpha_* (n+2)+(\alpha_*^2+2\alpha_*)(2n+k+6)\right)\int_{\mathbb{S}^n}|X|^{2\alpha_*-2}|\nabla u|^2 u^2 dV_k.
\end{split}
\end{equation}
According to the conditions (\ref{a.4}), it can be deduced that the right hand side of (\ref{a.10}) is less than or equal to $0$. Hence $\nabla u\equiv 0$ on $\mathbb{S}^n$, and then $M$ is an origin-centred sphere. We complete the proof of Theorem \ref{mainthm-a}.
\end{proof}

$\ $

\end{document}